\documentclass[12pt]{amsart}
\usepackage{graphicx} 
\usepackage{subcaption}
\usepackage{amsfonts}
\usepackage{amsmath}
\usepackage{float}
\usepackage{url}
\usepackage[round, sort, numbers]{natbib}
\usepackage{hyperref}
\usepackage{amsthm}

\usepackage[margin=2cm]{geometry}
\setcitestyle{square}

\makeatletter
\newcommand*{\rom}[1]{\expandafter\@slowromancap\romannumeral #1@}
\makeatother

\newtheorem{theorem}{Theorem}[section]
\newtheorem{prop}[theorem]{Proposition}

\newtheorem{lemma}[theorem]{Lemma}

\newtheorem{defin}{Definition}[section]
\newtheorem{remark}{Remark}[section]

\usepackage{color}

\title[Multiscaling in Wasserstein spaces]{Multiscaling in Wasserstein spaces}

\author[W. Mattar]{Wael Mattar}
\address[W. Mattar]{Tel Aviv, Israel}
\email{\tt waelmattar@tauex.tau.ac.il}

\author[N. Sharon]{Nir Sharon}
\address[N. Sharon]{Tel Aviv, Israel}
\email{\tt nsharon@tauex.tau.ac.il}



\begin{document}

\maketitle

\section*{Abstract}
We present a novel multiscale framework for analyzing sequences of probability measures in Wasserstein spaces over Euclidean domains. Exploiting the intrinsic geometry of optimal transport, we construct a multiscale transform applicable to both absolutely continuous and discrete measures. Central to our approach is a refinement operator based on McCann’s interpolants, which preserves the geodesic structure of measure flows and serves as an upsampling mechanism. Building on this, we introduce the optimality number, a scalar that quantifies deviations of a sequence from Wasserstein geodesicity across scales, enabling the detection of irregular dynamics and anomalies. We establish key theoretical guarantees, including stability of the transform and geometric decay of coefficients, ensuring robustness and interpretability of the multiscale representation. Finally, we demonstrate the versatility of our methodology through numerical experiments: denoising and anomaly detection in Gaussian flows, analysis of point cloud dynamics under vector fields, and the multiscale characterization of neural network learning trajectories.

\section{Introduction}

The Wasserstein spaces of probability measures have emerged as fundamental objects in modern mathematics, bridging optimal transport, geometry, and functional analysis~\cite{ambrosio2008gradient, villani2021topics, santambrogio2015optimal, panaretos2020invitation}. Over the past two decades, the geometric structure of these spaces, particularly the formal Riemannian structure definition introduced by Otto~\cite{Otto31012001}, has enabled powerful tools for studying dynamics and evolution of probability measures. These advances have found growing relevance in fields ranging from image processing~\cite{rabin2011wasserstein} and deep learning~\cite{gulrajani2017improved} to data analysis~\cite{bigot2020statistical, peyre2019computational} and geophysics~\cite{yang2018application}. Moreover, Wasserstein spaces have proven valuable in cell biology~\cite{demetci2020gromov, klein2024genot, chen2025fast, banerjee2024efficient, schiebinger2019optimal}.

Modeling data as probability measures offers several advantages, particularly when the data possess geometric, spatial, or structural characteristics that traditional vector spaces fail to capture. Recent mathematical advancements in Wasserstein spaces have led to efficient computational algorithms, including manifold learning~\cite{hamm2025manifold}, regression~\cite{chen2023wasserstein}, and interpolation~\cite{chewi2021fast, fan2024conditional}. However, multiscale analysis within these spaces remains largely unexplored, leaving significant potential for both theoretical advances and practical applications.

Inspired by representing data on different scales, multiscale analysis has become ubiquitous in many data-driven tasks. These mathematical tools allow us to express sequences in a hierarchical structure, capturing features at various locations and scales. A classic example of multiscale analysis is the multiresolution framework introduced by wavelets~\cite{daubechies1992ten}. Subdivision schemes~\cite{dyn2002subdivision}, closely connected to wavelets, can likewise be used to achieve multiscale representations. In particular, refinement operators serve as upsampling operators, while downsampling operators perform the reverse operation, allowing transitions back and forth between scales~\cite{dyn2021linear, donoho1992interpolating}. Adaptations to Riemannian manifolds can be found in~\cite{wallner2020geometric, mattar2023pyramid, rahman2005multiscale}.

In this paper, we introduce a new multiscaling method suitable for analyzing sequences in Wasserstein spaces over Euclidean spaces. A multiscale representation of a sequence in a Wasserstein space is a pyramid consisting of a coarse approximation, in addition to a set of sequences of Borel measurable functions, which we call \emph{details}. The coarse approximation, together with the detail coefficients, can perfectly reconstruct the original sequences through the inverse multiscale transform. To this end, we adapt an elementary subdivision scheme to the metric spaces by exploiting McCann's interpolants~\cite{mccann1997convexity}. Our adaptation can be realized as a generalization to the transport subdivision schemes, recently introduced in~\cite{baccou2024subdivision}, because it is suitable not only for discrete measures, but also absolutely continuous ones. In addition to the refinement operator, we define two binary operators ``$\oplus$'' and ``$\ominus$'' that are analogous to scalar addition and subtraction, and play a fundamental role in multiscaling. Both operators utilize the theory of optimal transport, and we provide a detailed description of their computation.

Once the multiscale transform and its inverse are established, we introduce the \emph{optimality number} to quantify the deviation of a measure flow from optimality. This is a novel scalar that captures the extent to which the analyzed sequence deviates from geodesic flow in Wasserstein spaces. The optimality number accounts for not only global structure but also local geometric errors across scales, thereby providing a valuable tool for studying measure evolution in the Wasserstein space. Furthermore, this number can be redesigned to emphasize specific features of sequences, tailored to the requirements of the analysis task.

The proposed multiscaling framework can be applied to both absolutely continuous and discrete measures and can be readily adapted to sequences of mixed types. Our study also includes theoretical results. In particular, it turns out that the detail coefficients exhibit geometric decay across scales, provided that the analyzed sequence is sampled from an absolutely continuous curve in the Wasserstein space. We prove this theoretical result in addition to the stability of the inverse multiscale transform.

We conclude our paper with a section dedicated to numerical experiments and illustrations, complementing the theoretical results presented earlier. We analyze sequences of different types via our multiscaling method. In particular, we analyze a synthetically-generated sequence of Gaussian measures and demonstrate the application of \emph{denoising} and \emph{anomaly detection}. We further analyze the evolution of a point cloud via a vector field, using an example from electromagnetism. Lastly, we show how multiscaling can be used to investigate learning trajectories of neural network, thus opening new research directions in deep learning. All results are reproducible via a package of Python code available online at \href{https://github.com/WaelMattar/Measures}{https://github.com/WaelMattar/Measures}.

The paper is organized as follows. Section~\ref{sec:preliminaries} provides the necessary knowledge from the optimal transport theory. Section~\ref{sec:Multiscaling_continuous} introduces the elementary multiscale transform exclusively for sequences of absolutely continuous measures. Next, Section~\ref{sec:multiscaling_discrete} adapts the multiscale transform to sequences of discrete measures, and discusses all the required technical modifications. Theoretical results that are suitable for the two cases of sequences appear afterward in Section~\ref{sec:theoretical_results}. Finally, Section~\ref{sec:numerical_illustrations} concludes the paper with 3 numerical demonstrations showing different aspects of multiscaling, including useful applications in various settings.


\newpage

\section{Preliminaries}\label{sec:preliminaries}

We review Wasserstein spaces and some of their properties.

\subsection{Wasserstein spaces}~\label{subsec:Wasserstein_intro}

Let $d\in\mathbb{N}$ and denote by $\mathcal{P}(\mathbb{R}^d)$ the set of all probability measures associated with the Borel $\sigma$-algebra induced by the standard topology of $\mathbb{R}^d$. The main object of interest in this work is the Wasserstein subspace $\mathcal{P}_p(\mathbb{R}^d)$, where $p\geq 1$, defined by
\begin{equation}~\label{eqn:wasserstein_space}
    \mathcal{P}_p(\mathbb{R}^d) = \bigg\{ \mu\in\mathcal{P}(\mathbb{R}^d) \; \big| \;\int_{\mathbb{R}^d}\|x\|^p d\mu(x) < \infty \bigg\}.
\end{equation}
Endowed with the Wasserstein distance function, which we will define next, the space $\mathcal{P}_p(\mathbb{R}^d)$ becomes a metric space.

We calculate the distance between two elements in $\mathcal{P}_p(\mathbb{R}^d)$ via the Wasserstein distance which we borrow from the optimal transport framework~\cite{santambrogio2015optimal}. To this end, we first define the functional $\mathcal{J}_p$ acting on a probability measure $\gamma$ over the product space $\mathbb{R}^d\times\mathbb{R}^d$ by
\begin{equation}~\label{eqn:wasserstein_functional}
    \mathcal{J}_p(\gamma) = \int_{\mathbb{R}^d\times \mathbb{R}^d}\|x-y\|^p d\gamma(x,y).
\end{equation}
In the terminology of optimal transport, the integrand of $\mathcal{J}_p$ is called the \emph{cost function}. For any probability measures $\mu,\nu\in\mathcal{P}_p(\mathbb{R}^d)$, the Wasserstein distance is defined via
\begin{equation}~\label{eqn:Wasserstein_distance}
    W_p^p(\mu, \nu) = \min_{\gamma\in\Pi(\mu, \nu)}\mathcal{J}_p(\gamma),
\end{equation}
where $\Pi(\mu,\nu)$ is the set of all joint measures with marginals $\mu$ and $\nu$. Particularly,
\begin{equation}~\label{eqn:joint_measures}
    \Pi(\mu,\nu)=\{\gamma\in\mathcal{P}(\mathbb{R}^d\times\mathbb{R}^d)\;|\;\; (\pi^x)_\#\gamma=\mu,\;\;(\pi^y)_\#\gamma=\nu\},
\end{equation}
where $\pi^x$ and $\pi^y$ denote the projection maps $\mathbb{R}^d\times\mathbb{R}^d\to\mathbb{R}^d$ on the $x$ and $y$ coordinates, respectively, while $\#$ denotes the pushforward operation. The set $\Pi(\mu,\nu)$ is nonempty; it contains the product measure $\mu\times\nu$.

The right-hand side of~\eqref{eqn:Wasserstein_distance} is called the Kantorovich optimization problem, and an element in $\Pi(\mu,\nu)$ is called a \emph{transport plan}. Because the cost function of~\eqref{eqn:wasserstein_functional} is a convex function of the Euclidean difference $x-y$, then for any measures $\mu,\nu\in\mathcal{P}_p(\mathbb{R}^d)$ there exists an optimal transport plan solving~\eqref{eqn:Wasserstein_distance}.

In the special case where $\mu$ is absolutely continuous with respect to the Lebesgue measure, then the optimization problem admits a unique solution supported on the graph of a function $T:\mathbb{R}^d\rightarrow\mathbb{R}^d$ called the \emph{Monge} map. In other words, the unique solution takes the form $(I, T)_{\#}\mu$ where $I$ denotes the identity map. Furthermore, the image measure of $\mu$ via $T$ is $\nu$. That is, $T_{\#}\mu=\nu$. For convenience, we denote the Monge map transporting $\mu$ to $\nu$ in this case by $T_\mu^\nu$. We note that this uniqueness holds for strictly convex costs, i.e., for $p > 1$. For $p = 1$, the Monge map is not necessarily unique; however, the multiscale framework presented in this work remains valid in this setting, provided that a Monge map is selected consistently throughout the transform, for instance via entropic regularization~\cite{peyre2019computational} or any other selection rule applied uniformly.

For the quadratic cost case, where $p=2$, the Monge map $T$ becomes the gradient of a convex function $u:\mathbb{R}^d\to\mathbb{R}^d$ provided that $\mu$ is absolutely continuous. For more detailed results see~\cite{santambrogio2015optimal, ambrosio2008gradient}.

\subsection{The formal Riemannian structure of Wasserstein spaces}

The Wasserstein space $\mathcal{P}_p(\mathbb{R}^d)$ exhibits many properties that are similar to a Riemannian manifold~\cite{panaretos2020invitation}. This fact has been first realized by Otto~\cite{Otto31012001} through looking at the \emph{continuity equation} as a mean to endow the Wasserstein space with a Riemannian-like structure. We will visit the continuity equation in detail later. For now we focus on McCann's~\cite{mccann1997convexity} constant-speed geodesics and define tangent spaces to $\mathcal{P}_p(\mathbb{R}^d)$.

Let $\mu_0, \mu_1\in\mathcal{P}_p(\mathbb{R}^d)$ be probability measures, and let $\gamma\in\Pi(\mu_0, \mu_1)$ be an optimal transport plan. For $t\in[0,1]$ define the map $\pi^t:\mathbb{R}^d\times\mathbb{R}^d\rightarrow\mathbb{R}^d$ by
\begin{equation}
    \pi^t(x,y) = (1-t)x+ty, \quad x,y\in\mathbb{R}^d.
\end{equation}
Then the curve $\{\mu_t=(\pi^t)_{\#}\gamma\}_{t\in[0,1]}$, known as the McCann's interpolant, is a constant-speed geodesic in $\mathcal{P}_p(\mathbb{R}^d)$ that connects $\mu_0$ to $\mu_1$. In particular, the following equality holds
\begin{equation}~\label{eqn:constant_speed}
    W_p(\mu_t, \mu_s) = (t-s)W_p(\mu_0, \mu_1), \quad 0\leq s\leq t\leq 1.
\end{equation}
We treat the element $\mu_t$ falling on a McCann's interpolant as the \emph{weighted average} between $\mu_0$ and $\mu_1$, and define the averaging operator $\mathfrak{M}$ by
\begin{equation}~\label{eqn:McCann_average}
    \mathfrak{M}(\mu_0, \mu_1; t) = (\pi^t)_{\#}\gamma, \quad t\in[0, 1],
\end{equation}
which outputs a measure in $\mathcal{P}_p(\mathbb{R}^d)$. When $p>1$ and at least one of $\mu_0$ or $\mu_1$ is absolutely continuous, the optimal transport plan $\gamma$ is uniquely determined by the Monge map $T_{\mu_0}^{\mu_1}$, and consequently $\mathfrak{M}$ is unique. In the general case, when neither measure is absolutely continuous, or when $p=1$, the optimal transport plan $\gamma$, and hence $\mathfrak{M}$, may not be unique; in this case a consistent selection rule must be specified. We elaborate on this point in Section~\ref{sec:multiscaling_discrete} in the context of discrete measures, where we discuss how to apply $\mathfrak{M}$ consistently across repeated iterations. Explicit formulas for computing $\mathfrak{M}$ are given for the cases where $\mu_0$ and $\mu_1$ are both absolutely continuous or discrete, see~\eqref{eqn:continuous_McCann_average} and~\eqref{eqn:discrete_McCann_average}.

It is shown in~\cite[Proposition 5.32]{santambrogio2015optimal} and~\cite[Chapter 7]{ambrosio2008gradient} that McCann's curves are the only constant-speed geodesics in $\mathcal{P}_p(\mathbb{R}^d)$. As a direct implication of this definition, it is reasonable to define the \emph{tangent space} of the Wasserstein space $\mathcal{P}_p(\mathbb{R}^d)$ at a probability measure $\mu\in\mathcal{P}_p(\mathbb{R}^d)$, see~\cite[Definition 8.5.1]{ambrosio2008gradient}, as
\begin{equation}~\label{eqn:tangent_space}
    \text{Tan}_\mu= \text{cl}_{L^p(\mu)}\{ s(T-I) \; | \; T=T_{\mu}^\nu \text{ for some } \nu\in\mathcal{P}_p(\mathbb{R}^d),\; s>0\},
\end{equation}
where the closure is done in the $L^p(\mu)$ function space. To be precise, the space $L^p(\mu)$ is comprised of maps $f:\mathbb{R}^d\rightarrow\mathbb{R}^d$ such that $\|f\|^p$ is integrable with respect to the measure $\mu$, but we write $L^p(\mu)$ for convenience. The tangent space is valid and linear for any measure $\mu\in\mathcal{P}_p(\mathbb{R}^d)$. For more details and alternative definitions, we refer to~\cite[Chapter 8]{ambrosio2008gradient}.

Using the tangent space definition~\eqref{eqn:tangent_space} we can then define the \emph{exponential map} at $\mu$, which we denote by $\text{Exp}_\mu:\text{Tan}_\mu\rightarrow\mathcal{P}_p(\mathbb{R}^d)$, explicitly by the formula
\begin{equation}~\label{eqn:exp_map}
    \text{Exp}_\mu\big(s(T-I)\big) = \big( s(T-I) + I \big)_\# \mu.
\end{equation}
It is easy to see, considering Section~\ref{subsec:Wasserstein_intro}, that if $\mu$ is absolutely continuous, then $\text{Exp}_\mu$ becomes surjective on $\mathcal{P}_p(\mathbb{R}^d)$. Particularly, the inverse \emph{logarithm map} $\text{Log}_\mu:\mathcal{P}_p(\mathbb{R}^d)\rightarrow \text{Tan}_\mu$ takes the form
\begin{equation}~\label{eqn:log_map}
    \text{Log}_\mu(\nu) = T_{\mu}^{\nu} - I,
\end{equation}
where $T_{\mu}^{\nu}$ is the Monge map between $\mu$ and an arbitrary probability measure $\nu\in\mathcal{P}_p(\mathbb{R}^d)$.

Recall that under these circumstances, the measure $(I, T_{\mu}^{\nu})_\#\mu$ is the unique optimal transport plan minimizing $\mathcal{J}_p$ of~\eqref{eqn:wasserstein_functional} and hence $\|\text{Log}_\mu(\nu)\|^p_{L^p(\mu)}=W^p_p(\mu, \nu)$ where $W_p(\mu, \nu)$ is the Wasserstein distance~\eqref{eqn:Wasserstein_distance}. Overall, with these notations, we can write $\text{Exp}_\mu(\text{Log}_\mu(\nu)) = \nu$ for any $\nu\in\mathcal{P}_p(\mathbb{R}^d)$, and $\text{Log}_\mu(\text{Exp}_\mu(s(T-I)))=s(T-I)$ for any $s\in[0, 1]$ and map $T$. These notions will become essential in the following sections.

\subsection{The continuity equation}

The Wasserstein space $\mathcal{P}_p(\mathbb{R}^d)$ can be endowed with a differential structure consistent with the formal Riemannian structure discussed earlier through the continuity equation. Here we review the essential mathematical tools to describe and study flows in $\mathcal{P}_p(\mathbb{R}^d)$.

Let us recall the definition of the metric derivative. Given an absolutely continuous curve $\{\mu_t\}_{t\in[0, 1]} \subset \mathcal{P}_p(\mathbb{R}^d)$, the metric derivative is defined by
\begin{equation}~\label{eqn:metric_derivative}
    |\mu^\prime|_t = \lim_{h\rightarrow 0}\frac{W_p(\mu_{t+h}, \mu_t)}{h},
\end{equation}
provided this limit exists. We recall the fact that Lipschitz curves are absolutely continuous.

It is shown in~\cite{ambrosio2008gradient}, see the note~\cite{santambrogio2010introduction} for a quick overview, that for any absolutely continuous curve $\{\mu_t\}_{t\in[0, 1]}$ there exists a Borel vector field $v_t:\mathbb{R}^d\rightarrow\mathbb{R}^d$ depending on $t\in[0,1]$ such that the continuity equation
\begin{equation}~\label{eqn:continuity_equation}
    \frac{\partial}{\partial t} \mu_t + \nabla \cdot (\mu_t v_t) = 0,
\end{equation}
is satisfied in $[0,1]\times\mathbb{R}^d$, and that $\|v_t\|_{L^p(\mu_t)}\leq |\mu^\prime|_t$ for almost every $t\in[0, 1]$ in the Lebesgue measure. Conversely, if the curve $\mu_t$ solves the continuity equation~\eqref{eqn:continuity_equation} for some Borel vector field $v_t$ with $\int_0^1\|v_t\|_{L^p(\mu_t)}dt<\infty$, then $\mu_t$ is absolutely continuous and $\|v_t\|_{L^p(\mu_t)}\geq |\mu^\prime|_t$ for almost every $t\in[0,1]$, see, e.g.,~\cite[Theorem~5.14]{santambrogio2015optimal}. Among all vector fields that produce the same flow $\mu_t$, there is a unique optimal one with smallest $L^p(\mu_t)$ norm, equal to the metric derivative,
\begin{equation}~\label{eqn:metric_velocity_equality}
    \|v_t\|_{L^p(\mu_t)}=|\mu^\prime|_t
\end{equation}
almost everywhere in $t$ that is termed the ``tangent'' vector field.

Vector fields solving~\eqref{eqn:continuity_equation} for a curve of measures $\mu_t$ are sometimes called velocity fields. The reason being that, if particles are distributed with the law $\mu_0$ and conform at each time $t$ to the velocity field $v_t$, then the position of all particles at time $t$ must reconstruct $\mu_t$. Numerical illustrations of such dynamics appear in Figures~\ref{fig:electric_pyramid_1} and~\ref{fig:electric_pyramid_2}.

We conclude this subsection with an evaluation of vector fields in the case where the curve $\mu_t$ has discrete values of $t$. Suppose that $t$ belongs to the values of the dyadic grid $2^{-\ell}\mathbb{Z}\cap[0, 1]$ for some $\ell\in\mathbb{N}$. Then, the curve $\mu_t$ contains $2^\ell + 1$ measures parametrized over $t=\{i 2^{-\ell} | i=0,\dots, 2^\ell\}$. Focusing on a consecutive pair of measures $\mu_{i 2^{-\ell}}$ and $\mu_{(i+1) 2^{-\ell}}$ we can consider an optimal transport map $T_i$ such that $(I, T_i)_\#\mu_{i 2^{-\ell}}$ minimizes $\mathcal{J}_p$ of~\eqref{eqn:wasserstein_functional}, where the minima is exactly the Wasserstein distance $W^p_p(\mu_{i 2^{-\ell}}, \mu_{(i+1) 2^{-\ell}})$ of~\eqref{eqn:Wasserstein_distance}. Therefore, we can call the map $v_{i 2^{-\ell}}:\mathbb{R}^d\rightarrow\mathbb{R}^d$ given by $v_{i 2^{-\ell}}(x)=(T_i(x)-x)/2^{-\ell}$ the ``\emph{discrete velocity field}'' at time $t=i 2^{-\ell}$. Consequently,
\begin{equation}~\label{eqn:norm_discrete_tangent_vector}
    \|v_{i 2^{-\ell}}\|_{L^p(\mu_{i 2^{-\ell}})}=\frac{W_p(\mu_{i 2^{-\ell}}, \mu_{(i+1) 2^{-\ell}})}{2^{-\ell}}.
\end{equation}
Notice how the right hand side approaches the metric derivative~\eqref{eqn:metric_derivative} as $\ell\rightarrow\infty$. Overall, the map $v_{i 2^{-\ell}}$ can be realized as the discrete tangent vector field of the sequence $\mu_t$ at $t=i 2^{-\ell}$.

\section{Multiscaling absolutely continuous measures}\label{sec:Multiscaling_continuous}

In this section, we first introduce the necessary operators needed to construct our multiscale transform, acting on sequences of absolutely continuous measures in $\mathcal{P}_p(\mathbb{R}^d)$. Next, we define the multiscale transform and describe it in detail, and then present the optimality number as a tool to determine ``how optimal'' measures flow in the metric space.

All measures in this section, unless stated otherwise, are assumed to be absolutely continuous. In Section~\ref{sec:multiscaling_discrete} we describe how to adapt all the notions and definitions to the case of discrete measures.

\subsection{The binary operators of subtraction and addition}

We introduce the binary ``minus'' operator that acts on two measures $\mu,\nu\in\mathcal{P}_p(\mathbb{R}^d)$ by
\begin{equation}~\label{eqn:ominus}
    \nu \ominus \mu = T_\mu^\nu - I,
\end{equation}
where $I$ is the identity map of $\mathbb{R}^d$. The outcome of $\ominus$ defines a unique measurable map from $\mathbb{R}^d$ to itself, which is exactly the unique Monge map from $\mu$ to $\nu$ but with a translation of $I$. Moreover, because $T_\mu^\mu=I$ for any measure $\mu$, we have $\mu\ominus\mu=0$ the trivial zero map.

Notice that we can recover the probability measure $\nu$ if we have $\mu$ and the difference $\nu\ominus\mu$ at hand. To this end we define the ``plus'' operator by
\begin{equation}~\label{eqn:oplus}
    \mu \oplus \psi = (I + \psi)_\#\mu,
\end{equation}
for any probability measure $\mu\in\mathcal{P}_p(\mathbb{R}^d)$ and Borel measurable map $\psi:\mathbb{R}^d\to\mathbb{R}^d$. The operation $\oplus$ accepts probability measures in its first argument, measurable maps in its second argument, and returns measures that are necessarily probability measures. Both operators $\ominus$ and $\oplus$ are well defined under the absolute continuity assumption, and they are \emph{compatible} in the sense that
\begin{equation}~\label{eqn:operators_compatibility}
    \mu \oplus (\nu \ominus \mu) = \nu,
\end{equation}
for any probability measures $\mu,\nu\in\mathcal{P}_p(\mathbb{R}^d)$.

An immediate implication of~\eqref{eqn:ominus} is that the difference $\nu\ominus\mu$ lies in the tangent space $\text{Tan}_\mu$. In particular, take $s=1$ and the Monge map $T=T_\mu^\nu$, and substitute in the general term $s(T-I)$ appearing in~\eqref{eqn:tangent_space}. The result can be thought of as a tangent vector representing the map $T_\mu^\nu$ emanating from the point $\mu$, and hence the translation with the identity which corresponds to the origin of $\text{Tan}_\mu$. Likewise, the addition operation of~\eqref{eqn:oplus} can be seen as projecting the tangent vector $T_\mu^\nu-I\in\text{Tan}_\mu$ to $\mathcal{P}_p(\mathbb{R}^d)$ via the pushforward operation. In consistency with the Exp and Log operators of~\eqref{eqn:exp_map} and~\eqref{eqn:log_map} we can rewrite the subtraction and addition via
\begin{equation}~\label{eqn:operators_exp_log}
    \nu\ominus\mu = \text{Log}_\mu(\nu) \quad \text{and} \quad \mu\oplus \psi = \text{Exp}_\mu(\psi),
\end{equation}
for any two measures $\mu,\nu\in\mathcal{P}_p(\mathbb{R}^d)$ and a measurable map $\psi$.

Overall, for any two measures $\mu,\nu\in\mathcal{P}_p(\mathbb{R}^d)$ we get the following relation, which will later become useful for analysis.
\begin{equation}~\label{eqn:useful_relation}
    \mathcal{J}_p\big((I,T_{\mu}^\nu)_{\#}\mu\big) = \int_{\mathbb{R}^d}\|T_\mu^\nu(x) - x\|^p d\mu(x) = \|\nu \ominus \mu\|^p_{L^p(\mu)} = W_p^p(\mu, \nu),
\end{equation}
where $\mathcal{J}_p$ is the functional~\eqref{eqn:wasserstein_functional}, and $W_p$ is the Wasserstein distance~\eqref{eqn:Wasserstein_distance}. In the particular case where $\mu=\nu$ almost everywhere, then all the quantities in~\eqref{eqn:useful_relation} become $0$.

\subsection{Refinement operators}

We now exploit McCann's interpolants~\eqref{eqn:McCann_average} to introduce a new refinement operator acting on sequences in $\mathcal{P}_p(\mathbb{R}^d)$. A similar family called transport subdivision schemes was recently introduced in~\cite{baccou2024subdivision} exclusively for the discrete probability measures case and in~\cite{DynSharon2025} for complete metric spaces.

\begin{defin}~\label{defin:subdivision_scheme}
    The \emph{elementary} subdivision scheme $\mathcal{S}$ acting on a sequence of measures $\boldsymbol{\mu}=\{\mu_i\}_{i\in\mathbb{Z}}$ in $\mathcal{P}_p(\mathbb{R}^d)$ is defined by the rules
    \begin{equation}~\label{eqn:subdivision_scheme}
    \begin{cases}
        (\mathcal{S}\boldsymbol{\mu})_{2i} = \mu_i, \\
        (\mathcal{S}\boldsymbol{\mu})_{2i+1} = \mathfrak{M}(\mu_i, \mu_{i+1}; 1/2),
    \end{cases}
    \end{equation}
    for all $i\in\mathbb{Z}$, where $\mathfrak{M}$ is the averaging operator~\eqref{eqn:McCann_average}.
\end{defin}

Because all measures in this section are assumed to be absolutely continuous, we have a unique optimal Monge map $T_{\mu}^{\nu}$ transporting $\mu$ to $\nu$, hence the average appearing in~\eqref{eqn:subdivision_scheme} takes the form
\begin{equation}~\label{eqn:continuous_McCann_average}
    \mathfrak{M}(\mu, \nu; t) = \big(I+t(T_{\mu}^{\nu}-I)\big)_{\#}\mu, \quad t\in[0, 1].
\end{equation}
In other words, the weighted average in this case can be obtained by the classical linear interpolation between the identity $I$ and $T_{\mu}^{\nu}$.

We associate the resulted sequence $\{(\mathcal{S}\boldsymbol{\mu})_i\}$ with the half integers $i\in 2^{-1}\mathbb{Z}$. The subdivision scheme $\mathcal{S}$ is interpolating in the sense that it preserves the original measures $\{\mu_i\}_{i\in\mathbb{Z}}$. The main purpose of subdivision schemes is to produce continuous (preferably smooth) curves from a discrete set of data points~\cite{dyn2002subdivision}. We will show below that the elementary subdivision scheme introduced in Definition~\ref{defin:subdivision_scheme} yields continuous curves. However, more sophisticated rules can be designed to yield smooth curves in Wasserstein spaces, e.g., an adaptation of the non-interpolating B-spline subdivision schemes is achievable through iterative averaging~\cite{dyn2017manifold}. Furthermore, exploiting the Lane-Riesenfield algorithm~\cite{lane1980theoretical}, one can approximate the analogues of the B-spline subdivision schemes based on $\mathfrak{M}$. In general, advanced subdivision schemes can be derived via barycenters in the Wasserstein space~\cite{agueh2011barycenters}. A recent interpolation method for the Wasserstein space where $p=2$, based on the well-studied Euclidean B-splines, was proposed in~\cite{chewi2021fast}.

The following proposition shows that iterative refinement of absolutely continuous measures via $\mathcal{S}$ is consistent.

\begin{prop}~\label{prop:subdivision_closedness}
    Let $\mu_0$ be an absolutely continuous measure in $\mathcal{P}_p(\mathbb{R}^d)$, and let $\mu_1\in\mathcal{P}_p(\mathbb{R}^d)$ be an arbitrary measure. Denote by $\mu_{1/2}=\mathfrak{M}(\mu_0, \mu_1; 1/2)$ the midpoint between $\mu_0$ and $\mu_1$. Then
    \begin{equation*}
        \mathfrak{M}(\mu_0, \mu_1; 1/4) = \mathfrak{M}(\mu_0,\mu_{1/2}; 1/2),
    \end{equation*}
    and
    \begin{equation*}
        \mathfrak{M}(\mu_0, \mu_1; 3/4) = \mathfrak{M}(\mu_{1/2},\mu_1; 1/2).
    \end{equation*}
\end{prop}

\begin{proof}
    Since $\mu_0$ is assumed to be absolutely continuous, then there exists a unique Monge map $T_{\mu_0}^{\mu_1}$ pushing $\mu_0$ onto $\mu_1$. The midpoint between these measures is hence given by~\eqref{eqn:continuous_McCann_average} as
    \begin{equation*}
        \mu_{1/2} = \big(\frac{1}{2}I+\frac{1}{2}T_{\mu_0}^{\mu_1}\big)_{\#}\mu_0.
    \end{equation*}
    Here, the map $\frac{1}{2}I+\frac{1}{2}T_{\mu_0}^{\mu_1}$ pushing $\mu_0$ onto $\mu_{1/2}$, which we denote by $T_{\mu_0}^{\mu_{1/2}}$, is not a mere map, but the optimal transport. Consequently, it is algebraically evident that
    \begin{align*}
        \mathfrak{M}(\mu_0, \mu_1; 1/4) & = \big(\frac{3}{4}I+\frac{1}{4}T_{\mu_0}^{\mu_1}\big)_{\#}\mu_0 = \bigg(\frac{1}{2}I+\frac{1}{2}\big(\frac{1}{2}I+\frac{1}{2}T_{\mu_0}^{\mu_1}\big)\bigg)_{\#}\mu_0
        \\ & = \big(\frac{1}{2}I+\frac{1}{2}T_{\mu_0}^{\mu_{1/2}}\big)_{\#}\mu_0=\mathfrak{M}(\mu_0,\mu_{1/2}; 1/2).
    \end{align*}
    The second equality can be shown in a similar manner.
\end{proof}

Proposition~\ref{prop:subdivision_closedness} leads us to the following remark.

\begin{remark}
    One can define the operator $\mathcal{S}^r$, $r\in\mathbb{N}$ as the decomposition of $\mathcal{S}$ on itself $r$-many times. Furthermore, following Proposition~\ref{prop:subdivision_closedness}, we have that all the measures of the refined sequence $\{(\mathcal{S}^r\boldsymbol{\mu})_i\}$, associated with the indices $i\in 2^{-r}\mathbb{Z}$, fall on the piecewise geodesic interpolant
    \begin{equation*}
        \mu_s = \mathfrak{M}(\mu_{[s]}, \mu_{[s]+1}; \{s\}), \quad s\in\mathbb{R},
    \end{equation*}
    where $[\cdot]$ and $\{\cdot\}$ are the floor and fractional part functions, respectively. More on the convergence analysis of such scheme, see~\cite{DynSharon2025}.
\end{remark}

Although the subdivision scheme~\eqref{eqn:subdivision_scheme} may deserve a separate study on its own, including its convergence analysis, we focus on its use in multiscale transforms, as we will see next.

\subsection{Elementary multiscale transform}

Multiscale transforms usually involve refinement operators as tools to predict missing data. In our framework, we use the elementary subdivision scheme $\mathcal{S}$ of~\eqref{eqn:subdivision_scheme} to refine sequences of measures in $\mathcal{P}_p(\mathbb{R}^d)$. Here we introduce the elementary multiscaling transform of sequences of measures, and then define the \emph{optimality number}.

Let $\boldsymbol{\mu}^{(J)}=\{\mu^{(J)}_i\}_{i\in 2^{-J}\mathbb{Z}}$ be a sequence in $\mathcal{P}_p(\mathbb{R}^d)$. The \emph{elementary multiscale analysis} is defined by the following iterations
\begin{equation}~\label{eqn:elementary_multiscaling}
    \boldsymbol{\mu}^{(\ell-1)}=\mathcal{D}\boldsymbol{\mu}^{(\ell)}, \quad \boldsymbol{\psi}^{(\ell)}= \boldsymbol{\mu}^{(\ell)} \ominus \mathcal{S}\boldsymbol{\mu}^{(\ell-1)}, \quad \ell=1,\dots, J,
\end{equation}
where $\mathcal{D}$ is the elementary downsampling operator given by $(\mathcal{D}\boldsymbol{\mu}^{(\ell)})_i=\mu^{(\ell)}_{2i}$ for any $i\in\mathbb{Z}$, while the difference operator $\ominus$ of~\eqref{eqn:ominus} is applied element-wise.

The analysis of the sequence $\boldsymbol{\mu}^{(J)}$ yields a pyramid $\{\boldsymbol{\mu}^{(0)}; \boldsymbol{\psi}^{(1)},\dots,\boldsymbol{\psi}^{(J)}\}$ that forms a representation to $\boldsymbol{\mu}^{(J)}$ on different scales. In particular, on the lowest scale we have a coarse approximation of measures, $\boldsymbol{\mu}^{(0)}\subset\mathcal{P}_p(\mathbb{R}^d)$, and $\boldsymbol{\psi}^{(\ell)}$, $\ell=1,\dots, J$ are the \emph{detail coefficients} of the analysis. Each sequence $\boldsymbol{\psi}^{(\ell)}$ encodes the measurable optimal transport maps between the elements of $\boldsymbol{\mu}^{(\ell)}$ and the predicted measures of the previous scale $\mathcal{S}\boldsymbol{\mu}^{(\ell-1)}$.

The pyramid of analysis can be synthesized back into $\boldsymbol{\mu}^{(J)}$ by iterating the addition operator~\eqref{eqn:oplus} as follows
\begin{equation}~\label{eqn:elementary_synthesis}
    \boldsymbol{\mu}^{(\ell)} = \mathcal{S}\boldsymbol{\mu}^{(\ell-1)}\oplus \boldsymbol{\psi}^{(\ell)}, \quad \ell=1,\dots, J.
\end{equation}
These iterations are called the \emph{inverse multiscale transform}, and they perfectly reconstruct $\boldsymbol{\mu}^{(J)}$ due to the compatibility condition~\eqref{eqn:operators_compatibility}.

Because the subdivision scheme $\mathcal{S}$ is interpolating, the detail coefficients generated by~\eqref{eqn:elementary_multiscaling} at all levels $\ell=1,\dots,J$, and all even indices $2i$, $i\in\mathbb{Z}$ must coincide with the trivial zero map. That is $\psi^{(\ell)}_{2i}=0$. Therefore, half of each layer of details $\boldsymbol{\psi}^{(\ell)}$ can be omitted when storing the pyramid representation. Overall, the number of nontrivial objects in the multiscale representation of $\boldsymbol{\mu}^{(J)}$ is equal to the number of measures in $\boldsymbol{\mu}^{(J)}$, and hence the multiscale transform~\eqref{eqn:elementary_multiscaling} can be used in practice for data compression. However, this is correct only when the operation $\ominus$ of~\eqref{eqn:ominus} requires no additional storage. For instance, when $\boldsymbol{\mu}^{(J)}$ is sampled from a family of distributions modeled with a fixed number of parameters, and the Monge maps between any two elements is guaranteed to have no greater number of parameters.

We briefly discuss the computational complexity of the elementary multiscale transform~\eqref{eqn:elementary_multiscaling}. Suppose the analyzed sequence $\boldsymbol{\mu}^{(J)}$ consists of $2^n$ measures, and that the transform is applied for $J \leq n$ levels. At each level $\ell = 1, \ldots, J$, the transform requires computing the detail coefficients $\boldsymbol{\psi}^{(\ell)}$, which amounts to evaluating the $\ominus$ operation between $\boldsymbol{\mu}^{(\ell)}$, with $2^{n+\ell-J}$ measures, and the Monge maps between consecutive measures in $\boldsymbol{\mu}^{(\ell-1)}$ as their predicted counterparts. Since $\ominus$ itself requires solving an optimal transport problem, we overall need to solve no greater than $2^{n+\ell-J}$ problems for each level $\ell$. Here, the computations of the even-indexed detail coefficients were skipped because they are the trivial zero map. Therefore, the total number of optimal transport problems solved across all $J$ levels is $\sum_{\ell=1}^{J} 2^{n-\ell+J} = 2^{n}(2^J-1)$. In practice, and in line with other multiscaling techniques, $J$ is typically small and does not exceed $6$, so the total number of optimal transport solves remains a moderate multiple of the sequence length $2^n$. In the absolutely continuous setting, each optimal transport problem reduces to solving a Monge--Amp\`ere equation, for which efficient numerical solvers are available, including entropic regularization approaches~\cite{peyre2019computational}. Overall, the transform is computationally tractable for sequences of moderate lengths, and its cost scales linearly with the sequence length, making it a practical tool for real-world applications.

Similar multiscale transforms can be established using different interpolating refinements by following the same formula of decompositions~\eqref{eqn:elementary_multiscaling}. For non-interpolating refinements however, the downsampling operator $\mathcal{D}$ needs to be modified in such a way to guarantee the property $\psi^{(\ell)}_{2i}=0$ for all $\ell=1,\dots,J$ and $i\in\mathbb{Z}$. In particular, it was shown in~\cite{dyn2021linear} that $\mathcal{D}\boldsymbol{\mu}^{(\ell)}$ must involve global averaging of the even elements of $\boldsymbol{\mu}^{(\ell)}$. Nevertheless, it was proven later in~\cite{mattar2023pyramid} that $\mathcal{D}$ can be approximated with local averaging at the expense of a controllable error manifested in $\psi^{(\ell)}_{2i}$.

We now define two norms that act on sequences of maps. Let $\boldsymbol{\mu}=\{\mu_i\}_{i\in\mathbb{Z}}\in\mathcal{P}_p(\mathbb{R}^d)$, and let $\boldsymbol{\psi}=\{\psi_i\}_{i\in\mathbb{Z}}$ be a corresponding sequence of measurable maps from $\mathbb{R}^d$ to itself. Then we define
\begin{equation}~\label{eqn:useful_norms}
    \|\boldsymbol{\psi}\|_{1} = \sum_{i\in\mathbb{Z}} \|\psi_i\|_{L^p(\mu_i)} \quad \text{and} \quad \|\boldsymbol{\psi}\|_{\infty} = \sup_{i\in\mathbb{Z}} \|\psi_i\|_{L^p(\mu_i)},
\end{equation}
where the norm $\|\cdot\|_{L^p(\mu_i)}$ is calculated as appears in~\eqref{eqn:useful_relation}. Although these norms exclude $\boldsymbol{\mu}$ from their notation, the sequence $\boldsymbol{\mu}$ can always be understood from the context.

Let us now introduce the optimality number. To determine how optimal sequences vary in $\mathcal{P}_p(\mathbb{R}^d)$, we treat the discrepancy between a general term $\mu_i^{(\ell)}$ and its predicted counterpart $(\mathcal{S}\boldsymbol{\mu}^{(\ell-1)})_{i}$ as an error. The significance of this error is calculated via the Wasserstein metric~\eqref{eqn:Wasserstein_distance}.

\begin{defin}~\label{defin:optimality_number}
    The optimality number $\omega$ of a sequence $\boldsymbol{\mu}^{(J)}\subset\mathcal{P}_p(\mathbb{R}^d)$ is defined by
    \begin{equation}~\label{eqn:optimality_number}
        \omega(\boldsymbol{\mu}^{(J)}) = \sum_{\ell=1}^J \|\boldsymbol{\psi}^{(\ell)}\|_{1},
    \end{equation}
    where $\boldsymbol{\psi}^{(\ell)}$, $\ell=1,\dots, J$ are the detail coefficients generated by the multiscale transform~\eqref{eqn:elementary_multiscaling}.
\end{defin}

The lower the value $\omega(\boldsymbol{\mu}^{(J)})$, the more optimal the flow of $\boldsymbol{\mu}^{(J)}$. Constant sequences and measures sampled along constant-speed geodesics have $0$ optimality, which is the best optimality number. In contrast, a sequence connecting two measures through a curve that deviates from their constant-speed geodesic would have a positive optimality number. If the deviation increases, then so does the optimality number. Overall, the value $\omega(\boldsymbol{\mu}^{(J)})$ can be used as a tool to indicate how optimal the sequence $\boldsymbol{\mu}^{(J)}$ flows in the Wasserstein spaces.

Observe that in Definition~\ref{defin:optimality_number}, the optimality number is defined in terms of the parameter $J$ associated with the analyzed sequence. In practice, this number is typically no greater than $6$. Moreover, the optimality number can be redefined to incorporate fewer levels of multiscale decompositions, depending on the desired coarse-scale approximation.

The benefit of calculating the optimality number~\eqref{eqn:optimality_number} via the multiscale analysis~\eqref{eqn:elementary_multiscaling} is that the detail coefficients describe the optimality errors on different scales and locations. The multiscale representation gives a clear image of both local and global errors. Moreover, if a sequence of measures is expected to evolve naturally, from an initial state to a final state, for instance, according to the optimal transport theory, then the pyramid transform applied to the observed sequence can reveal errors across scales and locations.

A drawback of defining $\omega$ as in~\eqref{eqn:optimality_number} is that detail coefficients associated with even indices do not contribute to $\omega$. This is due to the fact that $\|\psi^{(\ell)}_{2i}\|_{L^p(\mu^{(\ell)}_{2i})}=0$ for all $\ell=1,\dots,J$ and $i\in\mathbb{Z}$. However, this problem can be solved, for instance, by shifting the analyzed sequence $\boldsymbol{\mu}^{(J)}$ with one index to the left or right, compute the optimality as in~\eqref{eqn:optimality_number}, and then average with the original value $\omega(\boldsymbol{\mu}^{(J)})$.

Furthermore, the optimality number can be adjusted to reveal more information. For instance, one can penalize the $\ell$th layer of coefficients, $\|\boldsymbol{\psi}^{(\ell)}\|_1$, and multiply it with a factor, say $2^\ell$, to give more emphasis on changes that occur on high scales. Alternatively, the penalty could be applied more heavily to specific, predetermined regions over time. In short, the optimality number can be redesigned to capture valuable problem-specific information.

\section{Multiscaling discrete measures}\label{sec:multiscaling_discrete}

Here we treat the case where the sequences of interest consist of discrete measures. In particular, we revisit Section~\ref{sec:Multiscaling_continuous} and present the suitable modifications needed to adapt~\eqref{eqn:elementary_multiscaling} to the discrete case. The main differences lie in the averaging operator $\mathfrak{M}$ of~\eqref{eqn:McCann_average}, as well as the operations $\ominus$ and $\oplus$ of~\eqref{eqn:ominus} and~\eqref{eqn:oplus}, respectively.

Let $\nu,\mu\in\mathcal{P}_p(\mathbb{R}^d)$ be two discrete probability measures. Then, there exist $m,n\in\mathbb{N}$, and $m+n$ points $x^\mu_1,\dots, x^\mu_m, x^\nu_1, \dots, x^\nu_n\in\mathbb{R}^d$ such that
\begin{equation}~\label{eqn:arbitrary_discrete_measures}
    \mu = \sum_{i=1}^m p^\mu_{i}\delta_{x^\mu_i} \quad \text{and} \quad \nu=\sum_{j=1}^n p^\nu_j\delta_{x^\nu_j},
\end{equation}
where $\sum_{i=1}^m p^\mu_i=1$, $\sum_{j=1}^np^\nu_j=1$ and $p^\mu_i,p^\nu_j\geq0$. The measure $\delta$ here is Dirac's measure. Specifically, for any Borel set $\mathcal{A}\subseteq \mathbb{R}^d$ and a point $x\in\mathbb{R}^d$, we have $\delta_{x}(\mathcal{A})=1$ if $x\in\mathcal{A}$ and $\delta_{x}(\mathcal{A})=0$ otherwise.

The Kantorovich optimization problem~\eqref{eqn:Wasserstein_distance} reduces to solving the following linear program
\begin{equation}~\label{eqn:discrete_Kantorovich}
    \min_{\lambda_{i, j}} \;\sum_{i=1}^m\sum_{j=1}^n\big\|x^\mu_i-x^\nu_j\big\|^p\lambda_{i, j} \quad \text{subject to} \quad \sum_{i=1}^m\lambda_{i, j}=p^\nu_j, \quad \sum_{j=1}^n\lambda_{i, j}=p^\mu_i.
\end{equation}
The matrix solution $\Lambda_{\mu}^\nu=[\lambda_{i,j}]_{1\leq i\leq m, 1\leq j\leq n}$ is called the coupling matrix, where its entry $\lambda_{i, j}$ represents the amount of mass moving from the point $x^\mu_i$ to the point $x^\nu_j$. As stated in Section~\ref{sec:preliminaries}, because $p\geq 1$, there exists a solution to the problem. For the computational aspects of solving~\eqref{eqn:discrete_Kantorovich} we refer to~\cite{peyre2019computational}.

In the recent study~\cite{baccou2024subdivision}, the authors use the coupling matrix $\Lambda_\mu^\nu$ as a medium to construct an averaging operator, similar to our interpretation of McCann's average~\eqref{eqn:McCann_average}. As a result, this coupling-based operator is used to construct refinement rules similar to~\eqref{eqn:subdivision_scheme}. Here we follow the same methodology. Namely, the weighted average of the discrete measures~\eqref{eqn:arbitrary_discrete_measures} is given by
\begin{equation}~\label{eqn:discrete_McCann_average}
    \mathfrak{M}(\mu, \nu; t) = \sum_{i=1}^m\sum_{j=1}^n\lambda_{i,j}\delta_{x^t_{L(i, j)}},
\end{equation}
where $\lambda_{i,j}$ is the mass displaced from $x^\mu_i$ to $x^\nu_j$ via a coupling matrix, while the point $x^t_{L(i, j)}$ falls on the line segment connecting $x^\mu_i$ to $x^\nu_j$ with weight $t$. In particular, $x^t_{L(i,j)} = (1-t)x^\mu_i + tx^\nu_j$. It is intuitive to see that $\mathfrak{M}(\mu,\nu;0)=\mu$ and $\mathfrak{M}(\mu,\nu;1)=\nu$ due to the constrains of the Kantorovich problem~\eqref{eqn:discrete_Kantorovich}.

Equation~\eqref{eqn:discrete_McCann_average} can be realized as the analog of~\eqref{eqn:continuous_McCann_average} for the discrete measure case. Consequently, by using this explicit form of $\mathfrak{M}$, one can naturally obtain an analog version of the elementary subdivision scheme~\eqref{eqn:subdivision_scheme} that is suitable for refining sequences of discrete measures. Moreover, it was shown in~\cite{baccou2024subdivision} that the adaptation of the celebrated interpolating 4-point scheme~\cite{dyn1992subdivision}, and the non-interpolating corner-cutting scheme, via the averaging operator~\eqref{eqn:discrete_McCann_average}, are convergent.

For any $p\geq 1$, the coupling matrix $\Lambda_\mu^\nu$ solving~\eqref{eqn:discrete_Kantorovich} may not be unique, and hence the averaging operator $\mathfrak{M}$ of~\eqref{eqn:discrete_McCann_average} may not be unique either. When the subdivision scheme $\mathcal{S}$ is applied repeatedly, consistency across iterations must be ensured. Specifically, for any pair of consecutive measures $\mu_i^{(\ell)}$ and $\mu_{i+1}^{(\ell)}$, all intermediate averages produced at finer scales must lie on the same McCann's interpolant, i.e., the same constant-speed geodesic connecting $\mu_i^{(\ell)}$ to $\mu_{i+1}^{(\ell)}$. This is guaranteed by fixing a coupling matrix for each consecutive pair prior to the refinement process, and using it consistently across all subsequent iterations.

In practice, non-uniqueness arises only when the cost matrix exhibits a degenerate structure (e.g., exact symmetries or repeated costs), a situation that is unlikely to occur in real-world data due to noise and variability. However, uniqueness can always be enforced via entropic regularization~\cite{peyre2019computational}.

Moving forward, we now present the analogs of the operators $\ominus$ and $\oplus$ of~\eqref{eqn:ominus} and~\eqref{eqn:oplus} and how to compute them in practice. Let $\mu$ and $\nu$ be two discrete measures given by~\eqref{eqn:arbitrary_discrete_measures} and $\Lambda_{\mu}^{\nu}=[\lambda_{i,j}]_{1\leq i\leq m, 1\leq j\leq n}$ be a coupling matrix solving~\eqref{eqn:discrete_Kantorovich}. The difference operator is defined via
\begin{equation}~\label{eqn:discrete_ominus}
    \nu\ominus\mu=\bigg(\big[x^\nu_j-x^\mu_i\big]_{i=1,\dots,m, \;j=1,\dots,n}, \;\Lambda_\mu^\nu\bigg).
\end{equation}
The first argument of $\nu\ominus\mu$ is a tensor of order $3$ with $m$ rows, $n$ columns and $d$ slices, corresponding to the number of atoms of $\mu$ and $\nu$, and the coordinates of $\mathbb{R}^d$ respectively. The second argument encodes the coupling matrix between $\mu$ and $\nu$. In particular, the first argument can be geometrically described as all the vectors emanating from the points of $\mu$ to the points of $\nu$, along which a mass can be transported. The $m\times n$ coupling matrix $\Lambda_\mu^\nu$ is stored in the difference $\nu\ominus\mu$ for the purpose of perfectly reconstructing $\nu$ from $\mu$ and $\nu\ominus\mu$.

Conversely, let $\psi=(x^\psi,\Lambda^\psi)$ be a tuple consisting of a tensor $x^\psi$ of order $m\times k\times d$ for some $k\in\mathbb{N}$, and a matrix $\Lambda^\psi=[\lambda^\psi_{i,j}]$ of order $m\times k$ with nonnegative entries. We define $\oplus$ via
\begin{equation}~\label{eqn:discrete_oplus}
    \mu \oplus \psi = \sum_{i=1}^m\sum_{j=1}^k\lambda^\psi_{i, j}\delta_{x^\mu_i+x^\psi_{i,j}}.
\end{equation}
Put in simple words, the $\oplus$ operator distributes, or perhaps splits, the masses of its first argument according to the transport plan provided by its second one. We illustrate the computations of the difference between discrete measures in the following figure.
\begin{figure}[htbp]
  \centering
  \begin{subfigure}[b]{0.32\textwidth}
    \centering
    \includegraphics[width=\linewidth]{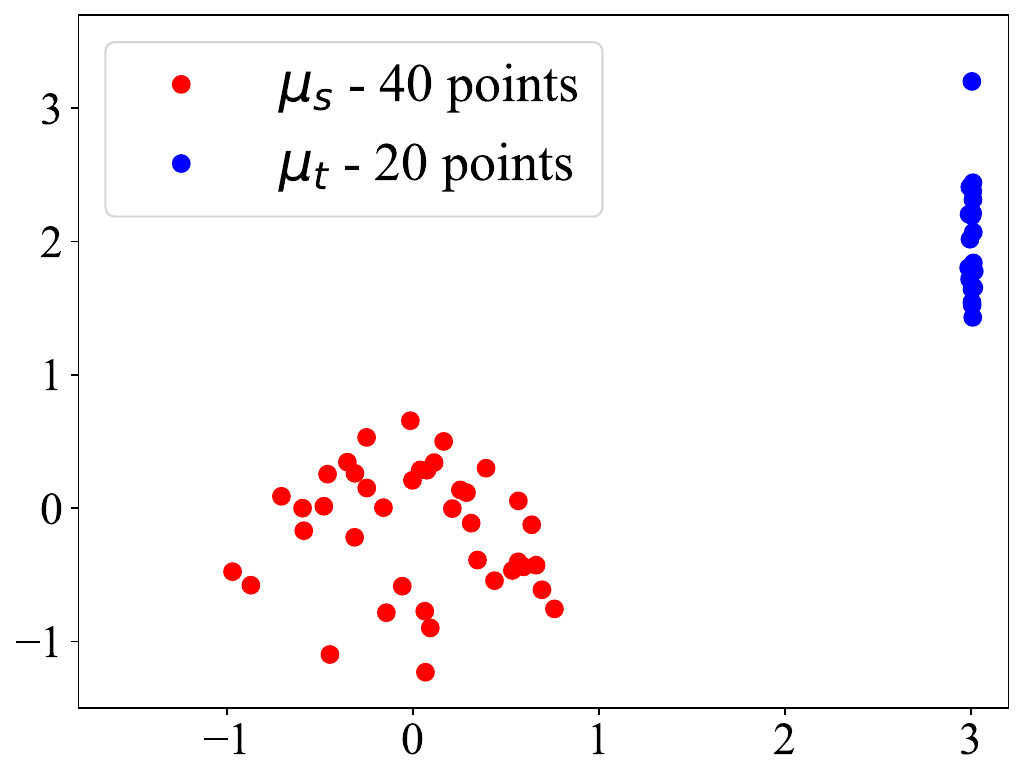}
  \end{subfigure}\hfill
  \begin{subfigure}[b]{0.32\textwidth}
    \centering
    \includegraphics[width=\linewidth]{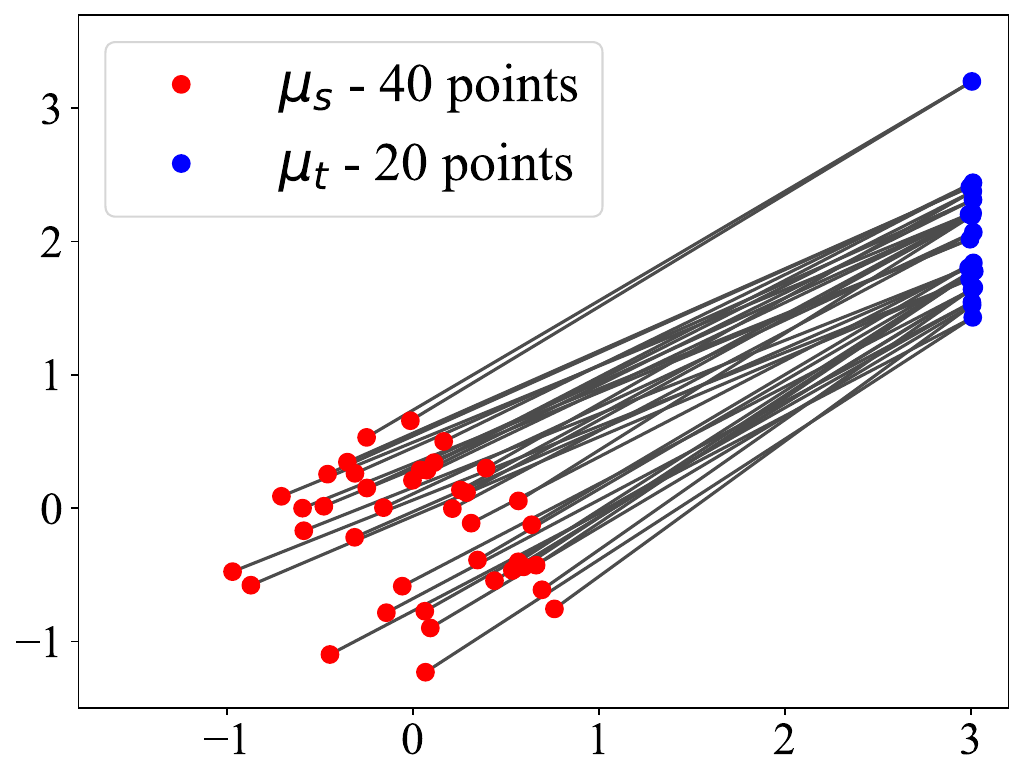}
  \end{subfigure}\hfill
  \begin{subfigure}[b]{0.32\textwidth}
    \centering
    \includegraphics[width=\linewidth]{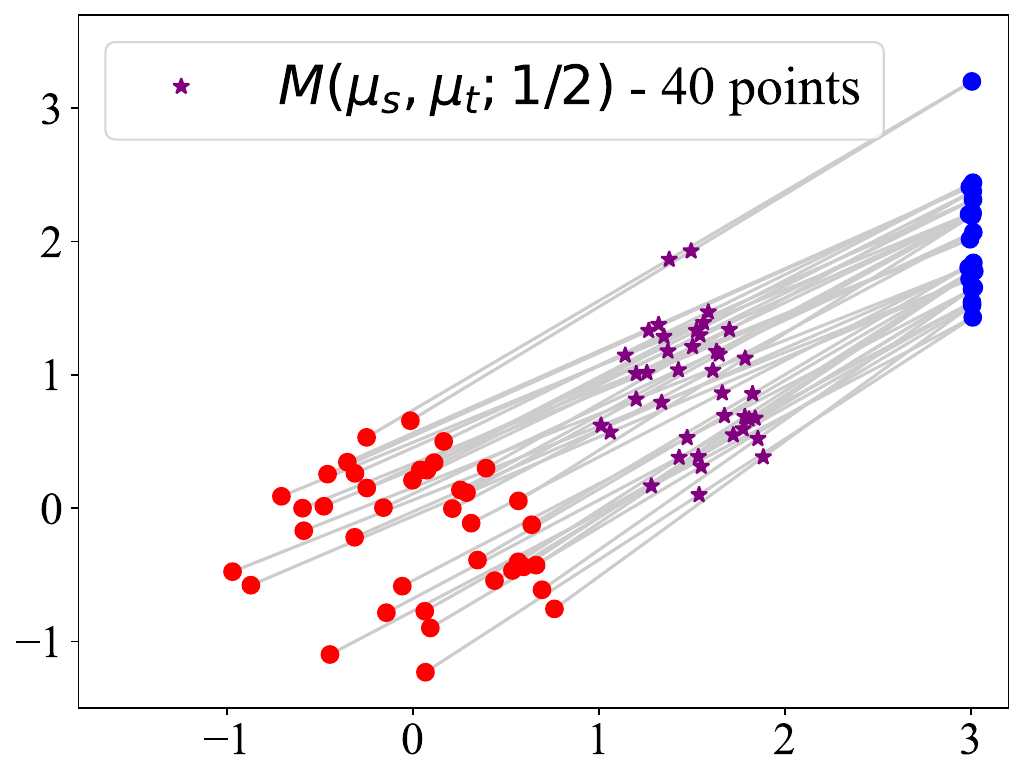}
  \end{subfigure}
  \caption{Illustration of the discrete $\ominus$ operator and McCann's average. On the left, the original source and target measures with uniform distribution over $40$ and $20$ points in $\mathbb{R}^2$, respectively. On the middle, the gray vectors depict the optimal transport plan for the quadratic cost between the measures. The difference $\ominus$ encodes these vectors in addition to the masses transported along each vector, $1/40$ in this case. On the right, McCann's average between the two measures.}
  \label{fig:operations_on_discrete_measures}
\end{figure}

We proceed with an insightful remark that will become essential in the following section.

\begin{remark}~\label{remark:psi_dichotomy}
    The addition operator~\eqref{eqn:discrete_oplus} that is suitable for discrete measures agrees with its counterpart~\eqref{eqn:oplus} in the following sense. Let $\mu\in\mathcal{P}_p(\mathbb{R}^d)$ be a discrete probability measure as in~\eqref{eqn:arbitrary_discrete_measures}, and let $\psi:\mathbb{R}^d\rightarrow\mathbb{R}^d$ be a measurable map. Then, the addition~\eqref{eqn:oplus} is well defined and becomes
    \begin{equation}~\label{eqn:alternative_oplus}
        \mu\oplus\psi \;=\;(I+\psi)_\#\mu \;= \sum_{y\;\in\;(I+\psi)(\{x^\mu_1,\dots,x^\mu_m\})}\bigg(\sum_{r\;:\; (I+\psi)(x^\mu_r)\;=\;y}p^\mu_r\bigg)\delta_{y}.
    \end{equation}
    In particular, if $I+\psi$ is an injective map, then the outer summation will run over $y=x^\mu_i+\psi(x^\mu_i)$ for $i=1,\dots, m$ while the inner summation will contain only one summand. Otherwise, the inner sum accounts for all the masses transported to the same point from different sources, as illustrated in Figure~\ref{fig:operations_on_discrete_measures}.
    
    The delicate equivalence between~\eqref{eqn:discrete_oplus} and~\eqref{eqn:alternative_oplus} reveals the dichotomous nature of $\psi$ and how it is possible to treat the difference $\ominus$ between two discrete measures. On the one hand, we can encode the difference between $\mu$ and $\nu$ in a practical way as the pair $\psi=(x^\psi, \Lambda^\psi)$ where $x^\psi$ is a 3-dimensional tensor as in~\eqref{eqn:discrete_ominus}. On the other hand, we can consider the outcome as a measurable map $\psi:\mathbb{R}^d\rightarrow\mathbb{R}^d$ interpolating the vectors of the tensor $x^\psi$, i.e., $\psi(x^\mu_i)=x^\nu_j$ for all $i=1,\dots,m$ and $j=1,\dots,n$, alongside the coupling matrix $\Lambda_\mu^\nu$ that tells us how much mass is transported from $x^\mu_i$ to $x^\nu_j$.
\end{remark}

Overall, the $\ominus$ operation of~\eqref{eqn:discrete_ominus} encodes the information for optimally transporting $\mu$ to $\nu$, while the $\oplus$ operation of~\eqref{eqn:discrete_oplus} takes $\mu$ and reconstructs $\nu$ according to the stored information, and therefore the compatibility condition~\eqref{eqn:operators_compatibility} holds for this construction. More importantly, the discrete version of the useful relation~\eqref{eqn:useful_relation} becomes
\begin{equation}
    \mathcal{J}_p(\Lambda_\mu^\nu) = \sum_{i=1}^m \sum_{j=1}^n \|x^\mu_i-x^\nu_j\|^p\lambda_{i,j}=\|\nu\ominus\mu\|^p_{\Lambda_\mu^\nu}=W^p_p(\mu, \nu),
\end{equation}
for the functional $\mathcal{J}_p$ of~\eqref{eqn:wasserstein_functional}, where $W_p$ is the Wasserstein distance~\eqref{eqn:Wasserstein_distance}, and the norm $\|\cdot\|^p_{\Lambda_\mu^\nu}$ is generally defined on $\psi=(x^\psi,\Lambda^\psi)$ by
\begin{equation}~\label{eqn:semi_norm}
   \|(x^\psi,\Lambda^\psi)\|^p_{\Lambda^\psi}=\sum_{i=1}^m\sum_{j=1}^k\|x^\psi_{i,j}\|^p\lambda^\psi_{i,j}.
\end{equation}

Algebraically, it is easy to see that the function $\|\cdot \|_{\Lambda^\psi}$ defines a semi-norm because it is nonnegative, homogeneous, and the triangle inequality is satisfied when the operations are defined on the tensor $x^\psi$, that is the first argument of $\ominus$. Although the expression~\eqref{eqn:semi_norm} can be zero for a nonzero pair $(x^\psi, \Lambda^\psi)$, e.g., when $\|x^\psi\|$ and $\Lambda^\psi$ have disjoint supports, we restrict the use of this norm to pairs that are obtained from the optimal transport theory. In particular, $x^\psi$ must encode vectors from some discrete measure to another, and $\Lambda^\psi$ is the coupling matrix between them solving~\eqref{eqn:discrete_Kantorovich}. This relation between the arguments guarantees, considering the properties of the Wasserstein distance, that $\|\nu\ominus\mu\|_{\Lambda_\mu^\nu}=0$ holds if and only if $\nu-\mu=0$. That is, $\mu=\nu$. Therefore, $\|\cdot\|_{\Lambda^\psi}$ becomes a norm under the restriction. In the terminology of optimal transport, the norm of the pair $(x^\psi,\Lambda^\psi)$ measures the total weighted displacement along the vectors $x^\psi$ according to the transport plan $\Lambda^\psi$.

Finally, multiscaling sequences of discrete measures is done in a similar fashion to~\eqref{eqn:elementary_multiscaling} where the operators involved are $\mathfrak{M}$ of~\eqref{eqn:discrete_McCann_average}, $\ominus$ and $\oplus$ of~\eqref{eqn:discrete_ominus} and~\eqref{eqn:discrete_oplus}. Therefore, all the discussions of Section~\ref{sec:Multiscaling_continuous} that are subsequent to~\eqref{eqn:elementary_multiscaling}, including the optimality number~\eqref{eqn:optimality_number}, extend naturally to the discrete case via the definitions provided in this section. In the next section, we study the properties of the multiscale transform.

\section{Theoretical results}\label{sec:theoretical_results}

In the spirit of classical results in harmonic analysis and wavelet theory, the two main theoretical results of this section establish that regularity of the analyzed sequence is reflected in the decay of its detail coefficients across scales, and that small perturbations to the pyramid representation yield proportionally small errors in the reconstructed sequence. In this section, we present our theoretical results that are suitable for the two cases: the case of absolutely continuous measures discussed in Section~\ref{sec:Multiscaling_continuous}, and the case of discrete measures discussed in Section~\ref{sec:multiscaling_discrete}. 

We use the norm notation $\|\cdot\|_{L^p(\mu)}$ appearing in~\eqref{eqn:useful_relation} to formalize our results in a general manner. That is, if the analyzed measures are discrete, then the theorems hold true when the norm is replaced with $\|(\cdot, \;\Lambda^\psi)\|_{\Lambda^\psi}$ of~\eqref{eqn:semi_norm}, when $\Lambda^\psi$ is understood from the context. Furthermore, the notations in~\eqref{eqn:useful_norms} are adapted to sequences of discrete measures via~\eqref{eqn:semi_norm} in a natural manner.

We first define the operator $\Delta$ acting on sequences of measures. Let $\boldsymbol{\mu}=\{\mu_i\}_{i\in\mathbb{Z}}\in\mathcal{P}_p(\mathbb{R}^d)$, then
\begin{equation}~\label{eqn:Delta}
    \Delta\boldsymbol{\mu} = \sup_{i\in\mathbb{Z}}W_p(\mu_i, \mu_{i+1}).
\end{equation}
The following lemma provides an estimate on the detail coefficients of~\eqref{eqn:elementary_multiscaling} that will become essential in main results.

\begin{lemma}~\label{lem:details_delta_bound}
     Let $\boldsymbol{\mu}^{(J)}$ be a sequence in $\mathcal{P}_p(\mathbb{R}^d)$ associated with the grid $2^{-J}\mathbb{Z}$. Then the detail coefficients $\boldsymbol{\psi}^{(\ell)}$ generated by the elementary multiscale transform~\eqref{eqn:elementary_multiscaling} satisfy
     \begin{equation}~\label{eqn:details_delta_bound}
         \|\boldsymbol{\psi}^{(\ell)}\|_{\infty} \leq 2\Delta \boldsymbol{\mu}^{(\ell)}, \quad \ell=1,\dots,J.
     \end{equation}
\end{lemma}

\begin{proof}
    The elements $\psi^{(\ell)}_{2i}$ are equal to the trivial zero map for all $\ell=1,\dots,J$ and $i\in\mathbb{Z}$. Therefore, $\|\psi^{(\ell)}_{2i}\|_{L^p(\mu^{(\ell)}_{2i})}=0$. Direct calculations of a general term $\psi^{(\ell)}_{2i+1}$ associated with an odd index give
    \begin{align*}
        \psi^{(\ell)}_{2i+1} & = \mu^{(\ell)}_{2i+1} \ominus \mathfrak{M}\big(\mu^{(\ell-1)}_{i}, \mu^{(\ell-1)}_{i+1}; \frac{1}{2}\big) = \mu^{(\ell)}_{2i+1} \ominus \mathfrak{M}\big(\mu^{(\ell)}_{2i}, \mu^{(\ell)}_{2i+2}; \frac{1}{2}\big).
    \end{align*}
    Consequently, by~\eqref{eqn:useful_relation} we get
    \begin{align*}
        \|\psi^{(\ell)}_{2i+1}\|_{L^p(\mu^{(\ell)}_{2i+1})} & = W_p\big(\mu^{(\ell)}_{2i+1},\; \mathfrak{M}(\mu^{(\ell)}_{2i}, \mu^{(\ell)}_{2i+2}; \frac{1}{2})\big) \\
        & \leq W_p\big(\mu^{(\ell)}_{2i+1}, \mu^{(\ell)}_{2i}\big) + W_p\big(\mu^{(\ell)}_{2i},\;\mathfrak{M}(\mu^{(\ell)}_{2i}, \mu^{(\ell)}_{2i+2}; \frac{1}{2})\big) \\
        & \leq \Delta \boldsymbol{\mu}^{(\ell)} + \frac{1}{2} W_p(\mu^{(\ell)}_{2i}, \mu^{(\ell)}_{2i+2}) \leq 2 \Delta \boldsymbol{\mu}^{(\ell)}.
    \end{align*}
    The first inequality is due to the metric property of $W_p$, and the second inequality is due to the constant-speed property~\eqref{eqn:constant_speed}. Taking the supremum norm over $i\in\mathbb{Z}$ gives the required result.
\end{proof}

The following theorem is a direct implication of Lemma~\ref{lem:details_delta_bound} and provides a clearer bound on $\|\boldsymbol{\psi}^{(\ell)}\|_{\infty}$ that decays geometrically provided a priori on $\boldsymbol{\mu}^{(J)}$.

\begin{theorem}~\label{thm:details_decay}
    Let $\boldsymbol{\mu}=\{\mu_t\}_{t\in\mathbb{R}}$ be an absolutely continuous curve in $\mathcal{P}_p(\mathbb{R}^d)$ with a finite metric derivative~\eqref{eqn:metric_derivative}, that is, $\Gamma = \sup_{t\in\mathbb{R}}|\boldsymbol{\mu}^\prime|_t<\infty$. If the sequence $\boldsymbol{\mu}^{(J)}$ is sampled from $\boldsymbol{\mu}$ over the dyadic grid $2^{-J}\mathbb{Z}$, then
    \begin{equation}~\label{eqn:details_decay}
        \|\boldsymbol{\psi}^{(\ell)}\|_{\infty} \leq \Gamma 2^{1-\ell}, \quad \ell=1,\dots,J,
    \end{equation}
    where $\boldsymbol{\psi}^{(\ell)}$ are the detail coefficients generated by the elementary multiscale transform~\eqref{eqn:elementary_multiscaling}.
\end{theorem}

\begin{proof}
    $\boldsymbol{\mu}$ is an absolutely continuous curve, hence there exists a Borel vector field $v_t$ such that the continuity equation~\eqref{eqn:continuity_equation} is satisfied. Because $\boldsymbol{\mu}^{(J)}$ is sampled from $\boldsymbol{\mu}$ at $2^{-J}\mathbb{Z}$, then straightforward calculations joining~\eqref{eqn:norm_discrete_tangent_vector} and~\eqref{eqn:Delta} show that
    \begin{equation*}
        \Delta\boldsymbol{\mu}^{(J)}=\sup_{t\in\mathbb{Z}} 2^{-J} \|v_t\|_{L^p(\mu^{(J)}_t)} \leq 2^{-J} \sup_{t\in\mathbb{Z}} |\boldsymbol{\mu}^\prime|_t=2^{-J} \Gamma.
    \end{equation*}
    Moreover, since $\boldsymbol{\mu}^{(\ell-1)}$ is obtained by decimating $\boldsymbol{\mu}^{(\ell)}$ with $\mathcal{D}$ for every $\ell=1,\dots, J$, we have that $\Delta\boldsymbol{\mu}^{(\ell-1)}\leq 2 \Delta\boldsymbol{\mu}^{(\ell)}$. Iteratively, one concludes $\Delta\boldsymbol{\mu}^{(\ell)}\leq 2^{-\ell}\Gamma$. Combining this result with~\eqref{eqn:details_delta_bound} gives the required.
\end{proof}

Theorem~\ref{thm:details_decay} suggests that if a sequence of probability measures in $\mathcal{P}_p(\mathbb{R}^d)$ behaves according to a vector field with finite metric derivative, then the norms of its detail coefficients must decay geometrically with a factor less than (or equal to) $2$ at each level. Practically, the theorem can be used to determine whether a sequence of measures obeys, or flows according to, a given vector field. Conversely, the theorem is useful for studying vector fields through analyzing empirical sequences of measures, i.e., the theorem can reveal whether the pair solves~\eqref {eqn:continuity_equation}.

We now prove the stability of the reconstruction process. Firstly, we invoke two useful inequalities from the optimal transport theory~\cite{panaretos2020invitation, peyre2019computational, villani2021topics}. Observe that for any measures $\mu,\nu\in\mathcal{P}_p(\mathbb{R}^d)$ and measurable Lipschitz maps $\psi, \widetilde{\psi}:\mathbb{R}^d\to\mathbb{R}^d$, the inequalities
\begin{equation}~\label{eqn:stability_inequalities}
    W_p(\psi_\#\mu,\widetilde{\psi}_\#\mu) \leq \|\psi-\widetilde{\psi}\|_{L^p(\mu)} \quad \text{and} \quad W_p(\psi_\#\mu,\psi_\#\nu) \leq \|\psi\|_{\text{Lip}}W_p(\mu, \nu),
\end{equation}
are satisfied, where $\|\psi\|_{\text{Lip}}$ is the Lipschitz constant given by
\begin{equation}~\label{eqn:Lipschitz_norm}
    \|\psi\|_{\text{Lip}} = \sup_{x\neq y}\frac{\|\psi(x)-\psi(y)\|}{\|x-y\|}.
\end{equation}
We provide the proof of these inequalities in the appendix. Secondly, for sequences $\boldsymbol{\mu}$ and $\boldsymbol{\nu}$ in $\mathcal{P}_p(\mathbb{R}^d)$ we define $\mathcal{W}_p(\boldsymbol{\mu}, \boldsymbol{\nu})=\sup_{i\in\mathbb{Z}}W_p(\mu_i, \nu_i)$. That is, the supremum of pair-wise distances. Lastly, we define the stability condition for refinement rules.

\begin{defin}~\label{defin:subdivision_stability}
    We say that a refinement rule $\mathcal{S}$ is \emph{stable} if for every two sequences $\boldsymbol{\mu}$ and $\boldsymbol{\nu}$ in $\mathcal{P}_p(\mathbb{R}^d)$, there exists a constant $K>0$ such that
    \begin{equation}~\label{eqn:subdivision_stability}
    \mathcal{W}_p(\mathcal{S}\boldsymbol{\mu},\mathcal{S}\boldsymbol{\nu})\leq K\mathcal{W}_p(\boldsymbol{\mu}, \boldsymbol{\nu}).
    \end{equation}
\end{defin}

A similar stability condition has been studied in~\cite{grohs2010stability}, including refinements on manifolds. Showing that the subdivision scheme $\mathcal{S}$ of~\eqref{eqn:subdivision_scheme} is stable is not a trivial task. The constant $K$ may depend on the curvature of the space $\mathcal{P}_p(\mathbb{R}^d)$. However, it is reasonable to assume that $\mathcal{S}$ is stable for dense enough sequences. In particular, assume $\boldsymbol{\mu}$ and $\boldsymbol{\nu}$ are  sequences such that $\Delta\boldsymbol{\mu}, \Delta\boldsymbol{\nu}\leq \delta$ for some $\delta>0$. Then for the new refinement elements we get
\begin{align*}
    W_p\big((\mathcal{S}\boldsymbol{\mu})_{2i+1}, (\mathcal{S}\boldsymbol{\nu})_{2i+1}\big) & \leq W_p\big((\mathcal{S}\boldsymbol{\mu})_{2i+1}, \mu_{i}\big) + W_p\big(\mu_{i}, \nu_{i}\big) + W_p\big(\nu_{i}, (\mathcal{S}\boldsymbol{\nu})_{2i+1}\big)
    \\ & = \frac{1}{2}W_p\big(\mu_i, \mu_{i+1}\big) + W_p\big(\mu_i, \nu_i\big) + \frac{1}{2}W_p\big( \nu_i, \nu_{i+1}\big) \leq \delta + W_p\big(\mu_i, \nu_i\big).
\end{align*}
Hence $\mathcal{W}_p(\mathcal{S}\boldsymbol{\mu}, \mathcal{S}\boldsymbol{\nu})\leq \delta + \mathcal{W}_p(\boldsymbol{\mu}, \boldsymbol{\nu})$. Therefore, by assuming $\delta\leq (K-1)\mathcal{W}_p(\boldsymbol{\mu},\boldsymbol{\nu})$ is small enough we get stability of $\mathcal{S}$ with constant $K$ for the pair of sequences. We are now ready to present and prove the multiscale stability result.

\begin{theorem}~\label{thm:stability}
    Let $\{\boldsymbol{\mu}^{(0)}; \boldsymbol{\psi}^{(1)},\dots,\boldsymbol{\psi}^{(J)}\}$ and $\{\widetilde{\boldsymbol{\mu}}^{(0)}; \widetilde{\boldsymbol{\psi}}^{(1)},\dots,\widetilde{\boldsymbol{\psi}}^{(J)}\}$ be two pyramid representations of two sequences $\boldsymbol{\mu}^{(J)}$ and $\widetilde{\boldsymbol{\mu}}^{(J)}$ in $\mathcal{P}_p(\mathbb{R}^d)$, respectively. Assume that the detail coefficients are uniformly bounded in their Lipschitz norm, that is $\|\psi^{(\ell)}_i\|_{\text{Lip}}\leq C$ for all $\ell=1, \dots, J$ and $i\in\mathbb{Z}$. If the subdivision scheme $\mathcal{S}$ involved in multiscaling is stable with the constant $K$, then
    \begin{equation}~\label{eqn:stability}
        \mathcal{W}_p(\boldsymbol{\mu}^{(J)}, \widetilde{\boldsymbol{\mu}}^{(J)}) \leq L\bigg(\mathcal{W}_p(\boldsymbol{\mu}^{(0)}, \widetilde{\boldsymbol{\mu}}^{(0)})+\sum_{\ell=1}^J\|\boldsymbol{\psi}^{(\ell)}-\widetilde{\boldsymbol{\psi}}^{(\ell)}\|_\infty\bigg),
    \end{equation}
    where $L=1$ if $KC\leq1$ and $L=(KC)^J$ otherwise.
\end{theorem}

\begin{proof}
    Recall that the sequences $\boldsymbol{\mu}^{(J)}$ and $\widetilde{\boldsymbol{\mu}}^{(J)}$ are synthesized by their corresponding pyramid representations via~\eqref{eqn:elementary_synthesis}. Observe that for any $\ell=1,\dots,J$ we have
    \begin{align*}
        \mathcal{W}_p(\boldsymbol{\mu}^{(\ell)}, \widetilde{\boldsymbol{\mu}}^{(\ell)}) & = \mathcal{W}_p\big(\mathcal{S}\boldsymbol{\mu}^{(\ell-1)}\oplus\boldsymbol{\psi}^{(\ell)},\; \mathcal{S}\widetilde{\boldsymbol{\mu}}^{(\ell-1)}\oplus\widetilde{\boldsymbol{\psi}}^{(\ell)}\big) \\
        & \leq \mathcal{W}_p\big(\mathcal{S}\boldsymbol{\mu}^{(\ell-1)}\oplus\boldsymbol{\psi}^{(\ell)},\; \mathcal{S}\widetilde{\boldsymbol{\mu}}^{(\ell-1)}\oplus\boldsymbol{\psi}^{(\ell)}\big) + \mathcal{W}_p\big(\mathcal{S}\widetilde{\boldsymbol{\mu}}^{(\ell-1)}\oplus\boldsymbol{\psi}^{(\ell)},\; \mathcal{S}\widetilde{\boldsymbol{\mu}}^{(\ell-1)}\oplus\widetilde{\boldsymbol{\psi}}^{(\ell)}\big) \\
        & \leq KC \mathcal{W}_p(\boldsymbol{\mu}^{(\ell-1)},\widetilde{\boldsymbol{\mu}}^{(\ell-1)}) + \big\|\boldsymbol{\psi}^{(\ell)} -\widetilde{\boldsymbol{\psi}}^{(\ell)}\big\|_\infty,
    \end{align*}
    where the third line is obtained by~\eqref{eqn:stability_inequalities} and~\eqref{eqn:subdivision_stability}. Repeating this estimation $J-1$ many times starting from $\ell=J$ gives the required.
\end{proof}

Theorem~\ref{thm:stability} guarantees that changes in the detail coefficients yield to proportional errors in synthesis. This fact is useful for many applications since, usually, modifications are applied to the detail coefficients prior to reconstruction.

We eventually note here that, due to Remark~\ref{remark:psi_dichotomy}, the inequalities~\eqref{eqn:stability_inequalities} and Theorem~\ref{thm:stability} are still true in case the analyzed sequence consists of discrete measures. In particular, a detail coefficient in the discrete case can be treated as a function from $\mathbb{R}^d$ to itself, in addition to a coupling matrix. Although the choice of the function is arbitrary, the Lipschitz norm of $\psi=\mu\ominus\nu$ is uniquely determined by restricting the points $x$ and $y$ appearing in~\eqref{eqn:Lipschitz_norm} to the atoms of the source measure $\mu$. As a result, the mathematical developments for the discrete case proceed naturally.

It was shown recently~\cite{nilesweed2022minimax} that the Wasserstein distance $W_p$ can be upper and lower bounded in terms of weighted $\ell_p$ norms of the wavelet coefficients of the density mismatch, effectively characterizing the Wasserstein distance through a Besov norm. This connection suggests that the detail coefficients $\boldsymbol{\psi}^{(\ell)}$ of our multiscale transform, which measure the Wasserstein discrepancy between the analyzed sequence and its geodesic prediction at each scale, may admit an alternative characterization in terms of wavelet coefficients of the involved densities, potentially offering a route to approximating the transform without explicitly solving optimal transport problems at every scale.


\section{Numerical illustrations}\label{sec:numerical_illustrations}

In this section, we present our numerical illustrations covering three types of sequences; we begin with measures that are absolutely continuous, see Section~\ref{sec:Multiscaling_continuous}, and then move to two cases of discrete measures following Section~\ref{sec:multiscaling_discrete}.

\subsection{Curves of Gaussian measures}

Computing the optimal transport plan that minimizes~\eqref{eqn:wasserstein_functional} between two measures in $\mathcal{P}_2(\mathbb{R}^d)$ is typically a difficult task. However, in certain cases, an explicit solution is available. For example, in the one dimensional case $d=1$, the optimal plan becomes a monotone displacement between the distributions of the measures. This is true since the cost function in~\eqref{eqn:wasserstein_functional} is a convex function of the Euclidean distance, see~\cite{villani2021topics}. Another case in which the optimal transport plan takes a closed-form expression for the quadratic cost is the Gaussian case for any $d\geq1$. For simplicity, here we review the results for the one-dimensional case and use them to illustrate the multiscaling of sequences of Gaussian measures, including the application of denoising and anomaly detection via our method. Moreover, we compute the optimality number of some curves of Gaussian measures. All operations required by the multiscale transform can be carried out analytically, rendering the computational cost of these experiments extremely low. We note that all of the computations presented in this subsection can be naturally extended to higher-dimensional Gaussian measures as well as mixed Gaussian distributions, see~\cite[Chapter 1.6.3]{panaretos2020invitation} and~\cite{chen2018optimal}.

Let $\mu_i\sim\mathcal{N}(m_i, \sigma_i)$, $i=0,1$, be two measures with Gaussian distributions on $\mathbb{R}$ with the means $m_i\in\mathbb{R}$ and the variances $\sigma_i>0$, respectively. The Wasserstein distance~\eqref{eqn:Wasserstein_distance} between $\mu_0$ and $\mu_1$ takes the form
\begin{equation}~\label{eqn:Gaussian_Wasserstein_distance}
    W_2^2(\mu_0, \mu_1)=(m_0-m_1)^2 + (\sqrt{\sigma_0}-\sqrt{\sigma_1})^2.
\end{equation}
In particular, the optimal transport map $T_{\mu_0}^{\mu_1}$ that pushes $\mu_0$ onto $\mu_1$ is the affine map
\begin{equation}~\label{eqn:Gaussian_optimal_transport}
    T_{\mu_0}^{\mu_1}(x)=m_1+\sqrt{\frac{\sigma_1}{\sigma_0}}\big(x-m_0\big), \quad x\in\mathbb{R},
\end{equation}
where the optimal transport plan $(I,T_{\mu_0}^{\mu_1})_{\#}\mu_0$ is supported on the set $\{(x,T_{\mu_0}^{\mu_1}(x))\;|\;x\in\mathbb{R}\}$ which constitute an affine subspace of $\mathbb{R}^2$. The multivariate version of these results have been known since~\cite{dowson1982frechet}.

The difference operator $\ominus$ of~\eqref{eqn:ominus} in this case is the affine map
\begin{equation}~\label{eqn:Gaussian_ominus}
    (\mu_1 \ominus\mu_0)(x) = m_1+\sqrt{\frac{\sigma_1}{\sigma_0}}\big(x-m_0\big)-x, \quad x\in\mathbb{R}.
\end{equation}
If we denote the result $\psi(x)=(\mu_1\ominus\mu_0)(x)$, then the addition operator $\oplus$ of~\eqref{eqn:oplus} applied to $\mu_0$ and $\psi$ recovers the Gaussian measure $\mu_1$, that is, $\mu_0\oplus \psi=\mu_1$. Therefore, the operator $\oplus$ can be expressed in a simple closed form by inverting the affine map~\eqref{eqn:Gaussian_ominus}. For this, a system of two equations with two variables (mean and variance) with a unique solution is solved. This solution is as follows. Given $\mu_0\sim\mathcal{N}(m_0,\sigma_0)$ and an affine map $\psi(x)=Ax+B$, the measure $\mu_1=\mu_0\oplus \psi$ is Gaussian and determined by the parameters
\begin{equation}
    \mu_0\oplus \psi\sim\mathcal{N}\big(B+m_0(A+1),\;\sigma_0(A+1)^2\big).
\end{equation}
Overall, the two operators are well defined and compatible~\eqref{eqn:operators_compatibility} for any Gaussian measures.

An element of the geodesic $\{\mu_t\}$ that connects $\mu_0$ with $\mu_1$ and parametrized with $t\in[0, 1]$ is given by $\mu_t\sim\mathcal{N}(m_t, \sigma_t)$, where
\begin{equation}~\label{eqn:Gaussian_geodesic}
        m_t = (1-t)m_0+tm_1 \quad \text{and} \quad \sigma_t = \big(1 + t\big(\sqrt{\frac{\sigma_1}{\sigma_0}}-1\big)\big)^2\sigma_0.
\end{equation}
The mean of $\mu_t$ is the weighted average between $m_0$ and $m_1$, while its standard deviation grows (or shrinks) linearly with the factor $|\sqrt{\sigma_1}-\sqrt{\sigma_0}|$. Note that the geodesic $\{\mu_t\}$ interpolates the points $\mu_0$ and $\mu_1$ for $t=0, 1$, respectively. Furthermore, the measure $\mu_t$ in this case is interpreted as McCann's average~\eqref{eqn:continuous_McCann_average} with weight $t$. That is, $\mu_t=\mathfrak{M}(\mu_0,\mu_1;t)$.

Now that all the ingredients of the elementary multiscale trasform~\eqref{eqn:elementary_multiscaling} are available, we illustrate a pyramid representation of a Gaussian measure curve. To this end, we consider the two probability measures $\mu_0\sim\mathcal{N}(0, 1.884)$ and $\mu_1\sim\mathcal{N}(1,0.1084)$, and a synthetically-generated curve connecting them of which we denote by $\{\widehat{\mu}_t\}$, $t\in[0, 1]$. The parameters of $\{\widehat{\mu}_t\}$ vary smoothly with respect to $t$. Moreover, the measures in the vicinity of $t=0.5$ have relatively high variances. This was done to create a discrepancy between $\{\widehat{\mu}_t\}$ and the geodesic $\{{\mu}_t\}$ that inherently encodes the optimal transport between the endpoint measures.

Figure~\ref{fig:smooth_gaussian_curve_pyramid} illustrates the curve $\{\widehat{\mu}_t\}$ together with its pyramid representation on $4$ scales. The maximal norm of the detail coefficients $\boldsymbol{\psi}^{(\ell)}$ generated by~\eqref{eqn:elementary_multiscaling} decay very fast. This indicates the smoothness of the curve $\{\widehat{\mu}_t\}$ as Theorem~\ref{thm:details_decay} indicates. To further illustrate our multiscaling, we contaminate the curve $\{\widehat{\mu}_t\}$ with noise, both to the means and variances of its elements, that becomes less significant in the neighborhoods of the endpoints of the curve. Figure~\ref{fig:noisy_gaussian_curve_pyramid} shows the noisy curve next to its multiscale representation. This time, because the curve does not vary smoothly, the detail coefficients are large on high scales, and show no clear pattern of geometric decay.

\begin{figure}
\begin{subfigure}{0.46\textwidth}
    \includegraphics[width=\textwidth]{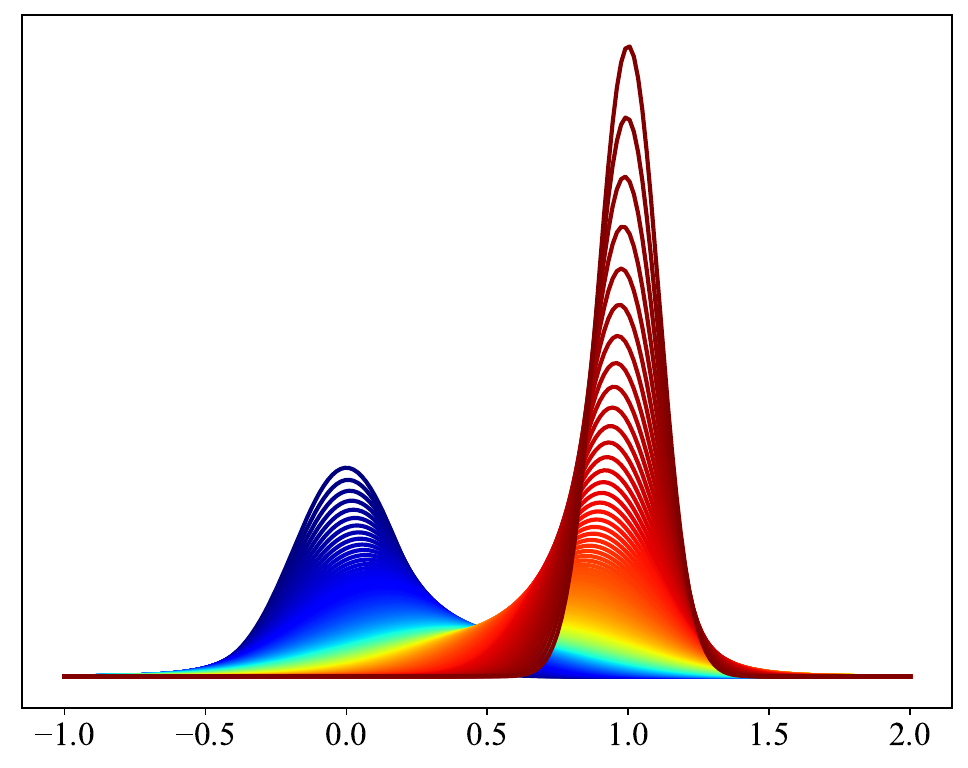}
\end{subfigure}
\hfill
\begin{subfigure}{0.51\textwidth}
    \includegraphics[width=\textwidth]{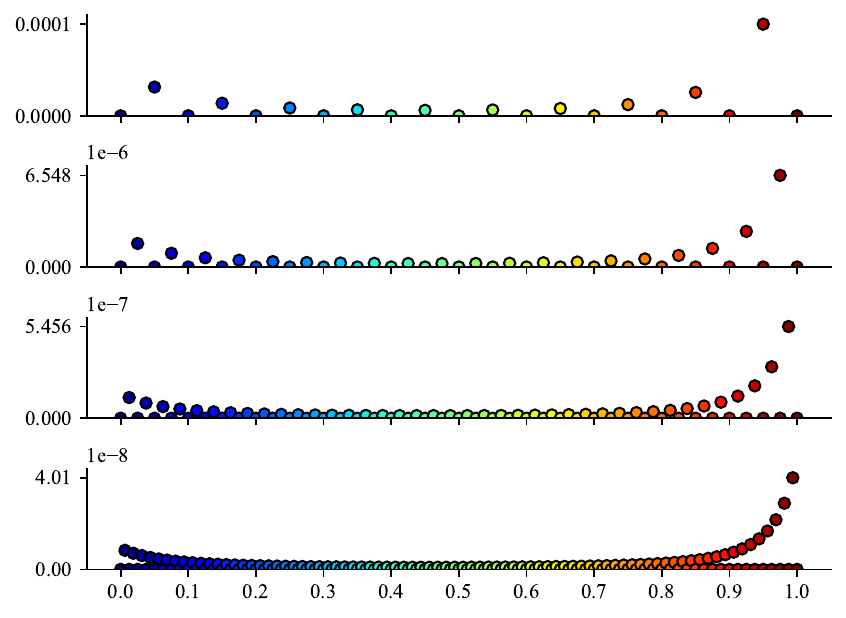}
\end{subfigure}
\caption{Analysis of the smooth Gaussian curve $\{\widehat{\mu}_t\}$. On the left, a curve of Gaussian measures with parameters that vary smoothly. On the right, norms of the detail coefficients $\boldsymbol{\psi}^{(\ell)}$ obtained by the elementary multiscale representation~\eqref{eqn:elementary_multiscaling}, where the top rows correspond to coarse scales and the bottom rows correspond to fine scales.} Note the decay of the maximal norm with each layer of details. The color coding in both figures correspond to each other.
\label{fig:smooth_gaussian_curve_pyramid}
\end{figure}

\begin{figure}
\begin{subfigure}{0.46\textwidth}
    \includegraphics[width=\textwidth]{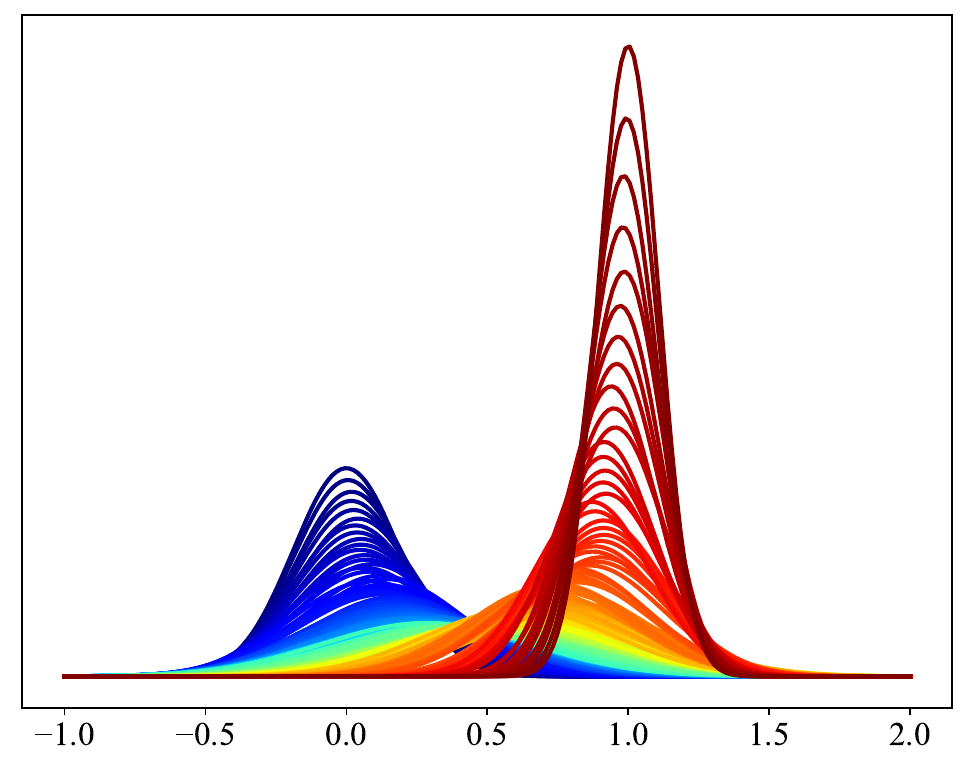}
\end{subfigure}
\hfill
\begin{subfigure}{0.51\textwidth}
    \includegraphics[width=\textwidth]{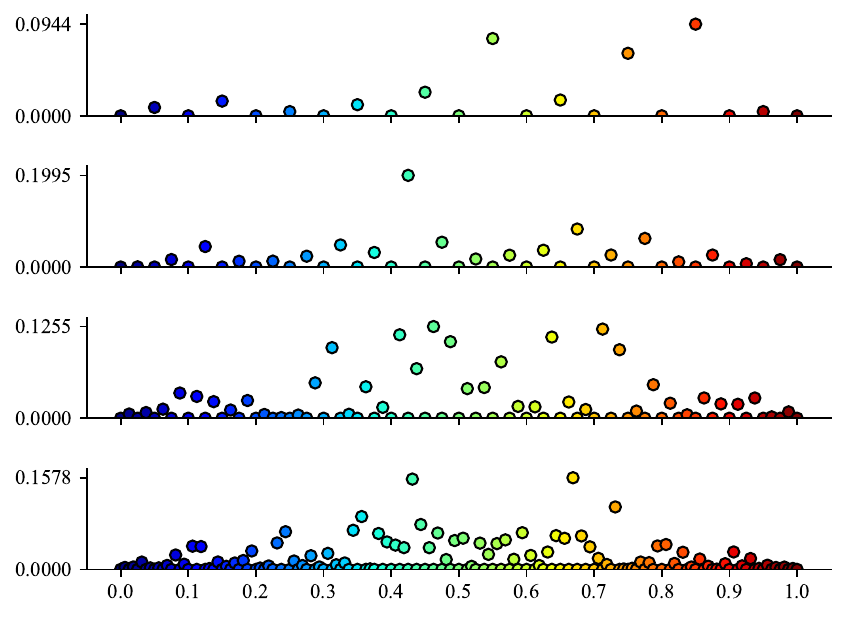}
\end{subfigure}
\caption{Analysis of a noisy Gaussian curve. On the left, the smooth curve of Gaussian measures $\{\widehat{\mu}_t\}$ but with parameters contaminated with noise. On the right, norms of the detail coefficients $\boldsymbol{\psi}^{(\ell)}$ obtained by the elementary multiscale representation~\eqref{eqn:elementary_multiscaling}. The norms show no geometric decay and remain high even at high scales. This indicates the noisy texture of the curve. The color codings in both figures correspond to each other.}
\label{fig:noisy_gaussian_curve_pyramid}
\end{figure}

Representing data on different scales is a powerful tool for applying denoising. We thus proceed and show the effect of denoising via our multiscaling. The application of noise reduction is done particularly by setting to zero detail coefficients with large norms, ones that are above a certain prefixed threshold. Here, by zero we mean the trivial zero map on $\mathbb{R}$. Thresholding the pyramid representation yields a sparser pyramid that can be reconstructed via~\eqref{eqn:elementary_synthesis} to obtain the denoised result. Figure~\ref{fig:denoised_gaussian_curve_pyramid} demonstrates the final result as a proof of concept.

\begin{figure}
\begin{subfigure}{0.46\textwidth}
    \includegraphics[width=\textwidth]{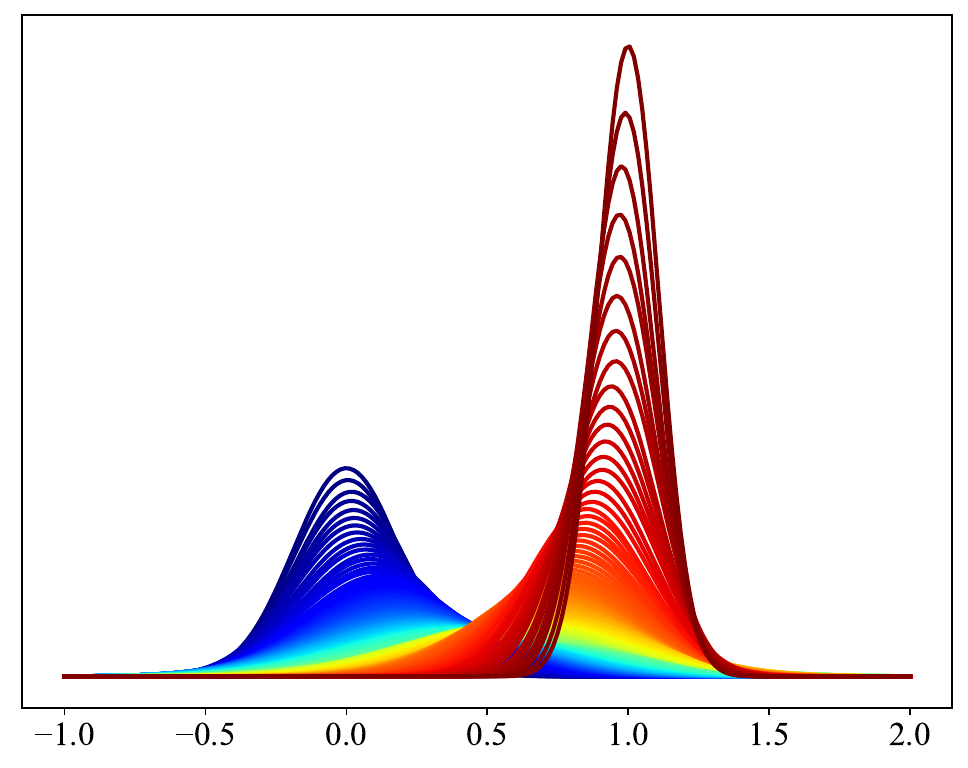}
\end{subfigure}
\hfill
\begin{subfigure}{0.51\textwidth}
    \includegraphics[width=\textwidth]{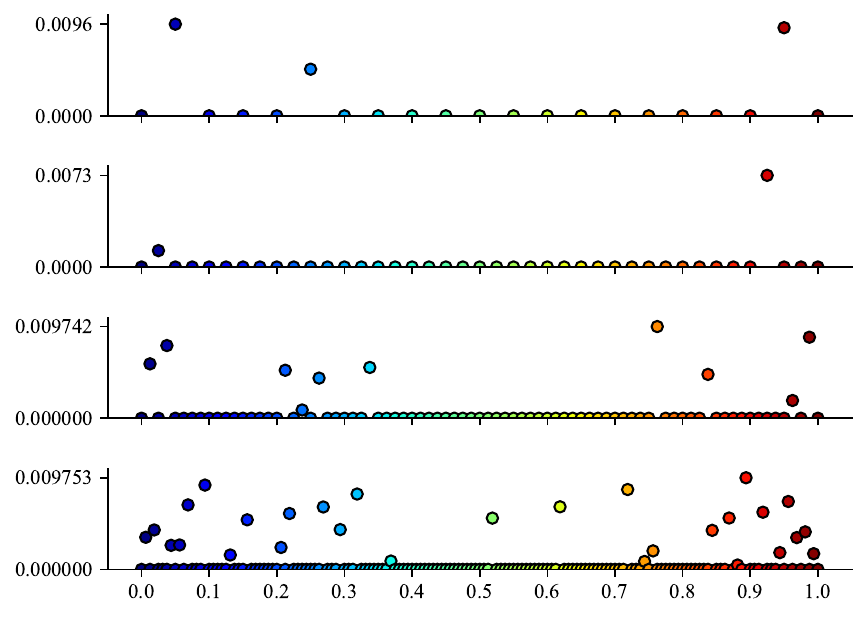}
\end{subfigure}
\caption{Denoised Gaussian curve. On the left, the result of denoising the curve that appears in Figure~\ref{fig:noisy_gaussian_curve_pyramid}, where the ground truth curve $\{\widehat{\mu}_t\}$ appears in Figure~\ref{fig:smooth_gaussian_curve_pyramid}. The denoising was done by thresholding the detail coefficients of the elementary multiscale transform with the threshold $0.01$. On the right, the detail coefficients of the denoised curve. The maximal Wasserstein distance between the original curve $\{\widehat{\mu}_t\}$ and the denoised curve is $0.1888$. Lowering the threshold gives better visual results with smaller empirical errors.}
\label{fig:denoised_gaussian_curve_pyramid}
\end{figure}

Another useful application that can be performed via multiscaling is \emph{anomaly detection}. In particular, this application is done by observing the significance of the details generated by the multiscale transform~\eqref{eqn:elementary_multiscaling}. Abnormalities such as jump discontinuities are detected in locations that correspond to relatively large detail coefficients. To illustrate this, we create two significant jump points in the middle of the smooth curve $\{\widehat{\mu}_t\}$ that appears in Figure~\ref{fig:smooth_gaussian_curve_pyramid}. Specifically, we drastically reduce the variances of the Gaussian measures falling in the middle third of the curve parametrization, hence creating two jump points. Indeed, the locations of these anomalies are revealed by large detail coefficients as Figure~\ref{fig:anomaly_gaussian_curve_pyramid} shows.

\begin{figure}
\begin{subfigure}{0.46\textwidth}
    \includegraphics[width=\textwidth]{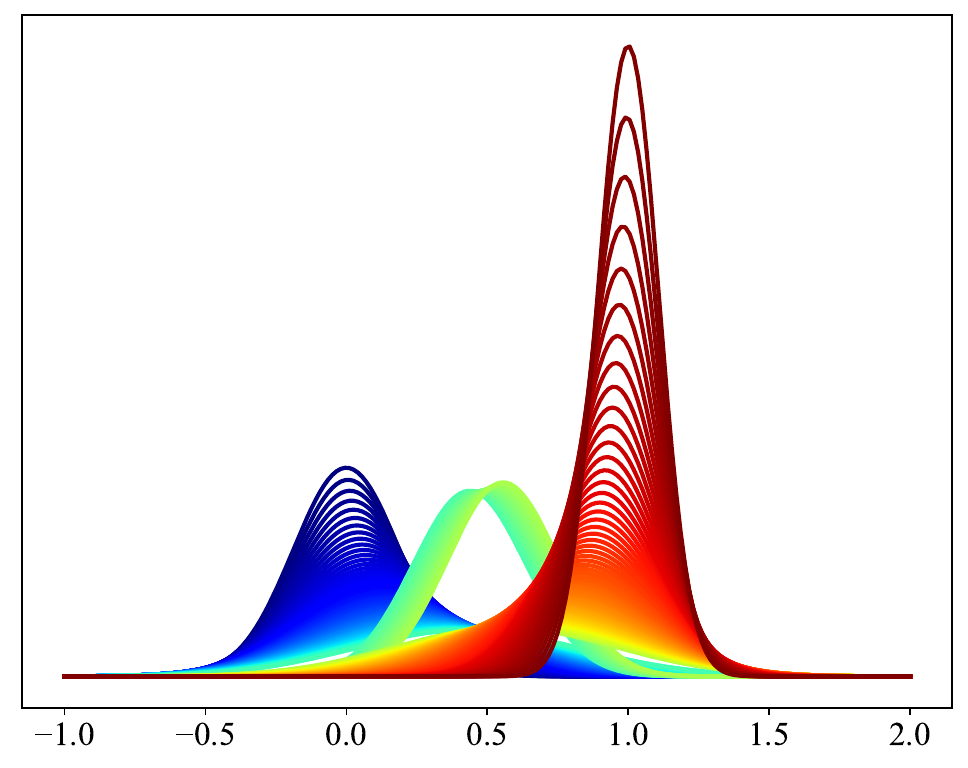}
\end{subfigure}
\hfill
\begin{subfigure}{0.51\textwidth}
    \includegraphics[width=\textwidth]{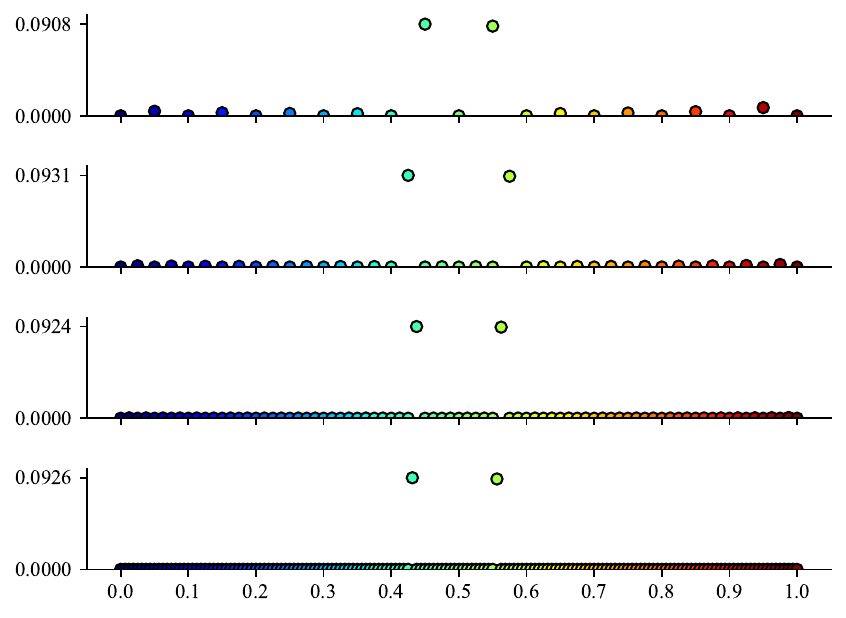}
\end{subfigure}
\caption{Anomaly detection in Gaussian curve. The locations of two jump discontinuities of a Gaussian curve are revealed by the elementary multiscale transform~\eqref{eqn:elementary_multiscaling}.}
\label{fig:anomaly_gaussian_curve_pyramid}
\end{figure}

Finally, we exploit the synthetic curve $\{\widehat{\mu}_t\}$ of Figure~\ref{fig:smooth_gaussian_curve_pyramid} to demonstrate how the optimality number increases as curves deviate from their geodesics in the Wasserstein space. To this purpose, by~\eqref{eqn:Gaussian_geodesic} we calculate the geodesic between the endpoint measures $\mu_0$ and $\mu_1$ that were given earlier in this section. Denote the geodesic by $\{\mu_t\}$, $t\in[0,1]$. Because the geodesic $\{\mu_t\}$ consists of intrinsic optimal transports between any two elements, of any location and scale, the detail coefficients of the elementary multiscale transform~\eqref{eqn:elementary_multiscaling} are all equal to the zero map. Therefore, the optimality number is $0$. i.e., the flow of $\{\mu_t\}$ is as optimal as possible.

Now, we compute the weighted averages between the geodesic $\{\mu_t\}$ and $\{\widehat{\mu}_t\}$ which both connect the initial and the final measures $\mu_0$ and $\mu_1$. Namely, define the family of curves $\mu^{[k]}$ by
\begin{equation}~\label{eqn:Gaussian_weighted_family}
    \mu^{[k]}_t = (1-k)\mu_t+k\widehat{\mu}_t, \quad (k, t)\in[0, 1]^2.
\end{equation}
Figure~\ref{fig:Gaussian_averages} depicts five members of this family, together with the optimality number of each curve. Furthermore, Figure~\ref{fig:Gaussian_averages_decay} shows the maximal norm of each detail layer for the five curves. The geometric decay therein indicates the smoothness of the curves, as pointed out in Theorem~\ref{thm:details_decay}.

\begin{figure}
\begin{subfigure}{0.19\textwidth}
    \includegraphics[width=\textwidth]{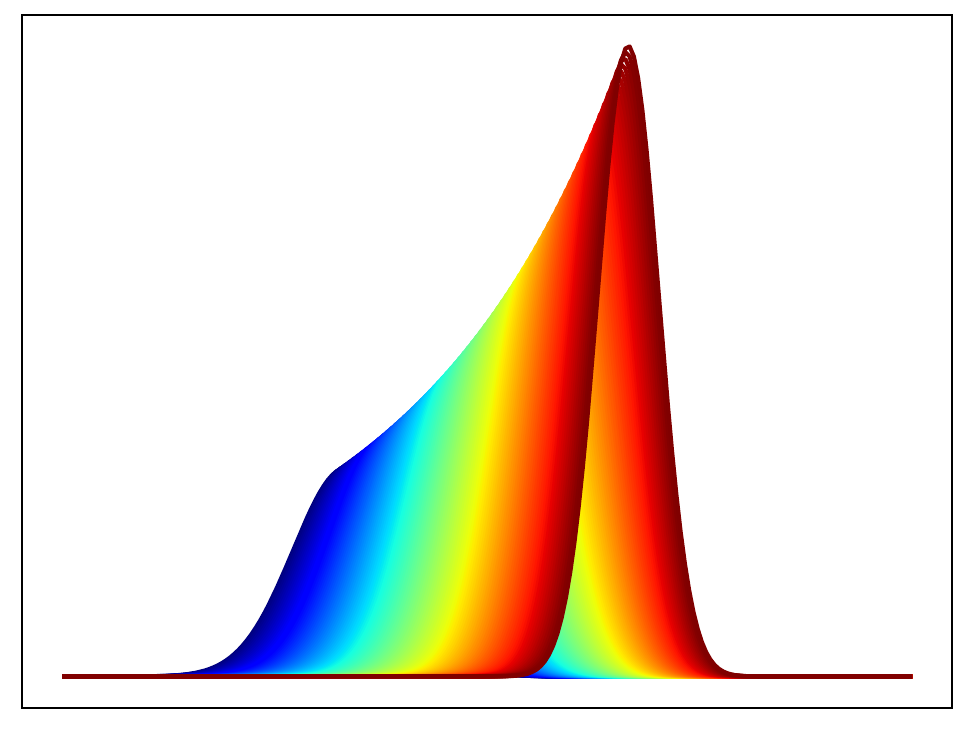}
    \caption*{$k=0$ \\ $\omega=0$}
\end{subfigure}
\hfill
\begin{subfigure}{0.19\textwidth}
    \includegraphics[width=\textwidth]{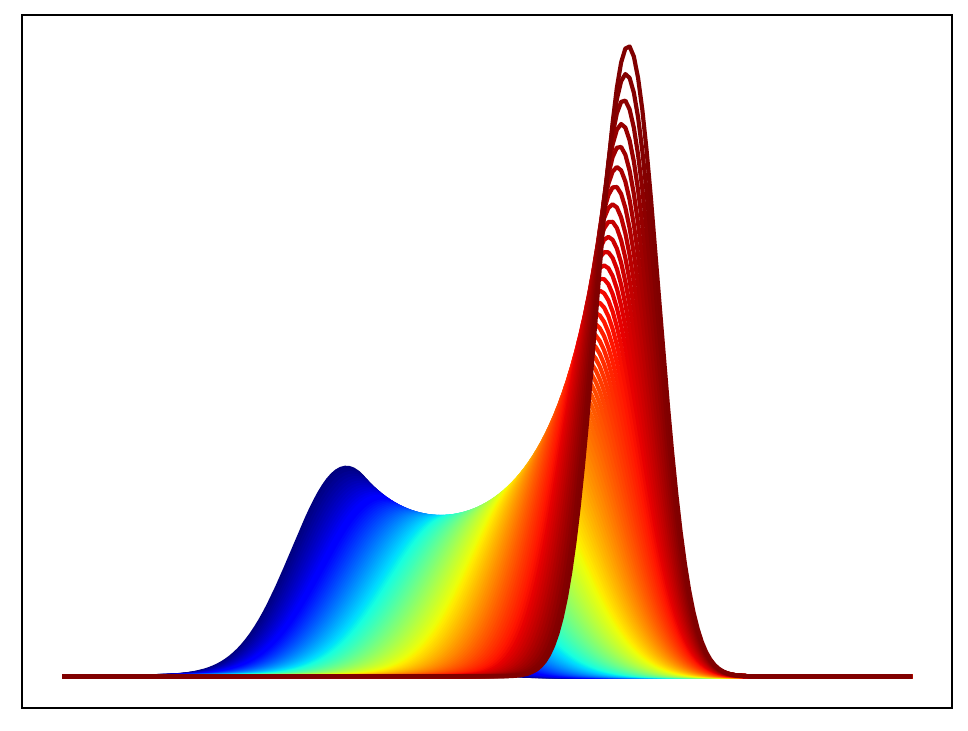}
    \caption*{$k=0.25$ \\ $\omega=0.5013$}
\end{subfigure}
\hfill
\begin{subfigure}{0.19\textwidth}
    \includegraphics[width=\textwidth]{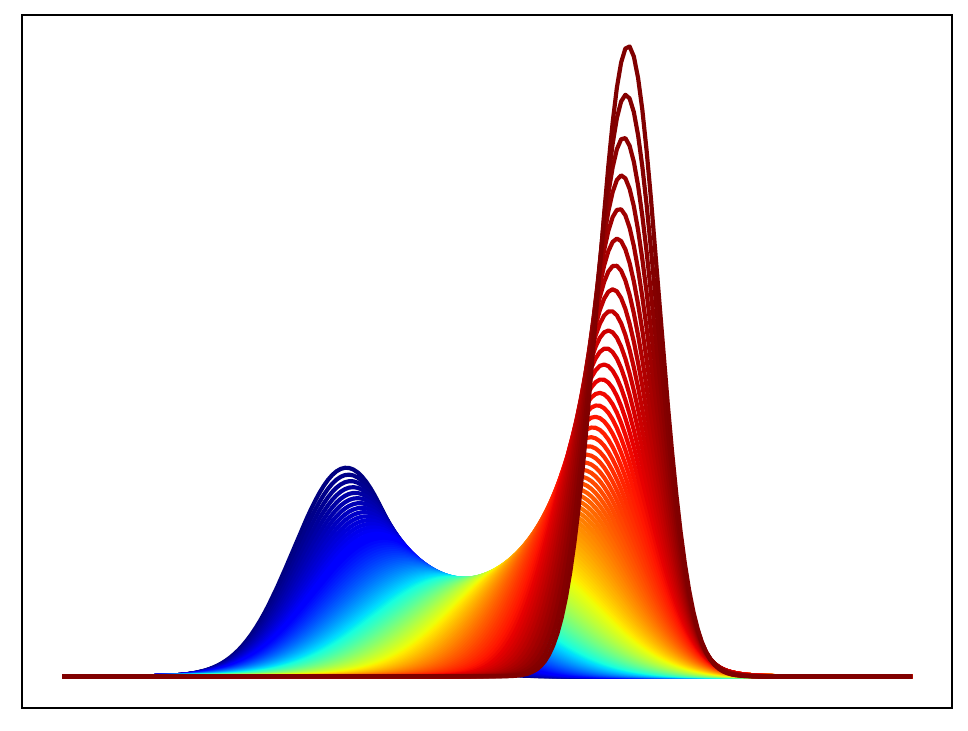}
    \caption*{$k=0.5$ \\ $\omega=0.9934$}
\end{subfigure}
\hfill
\begin{subfigure}{0.19\textwidth}
    \includegraphics[width=\textwidth]{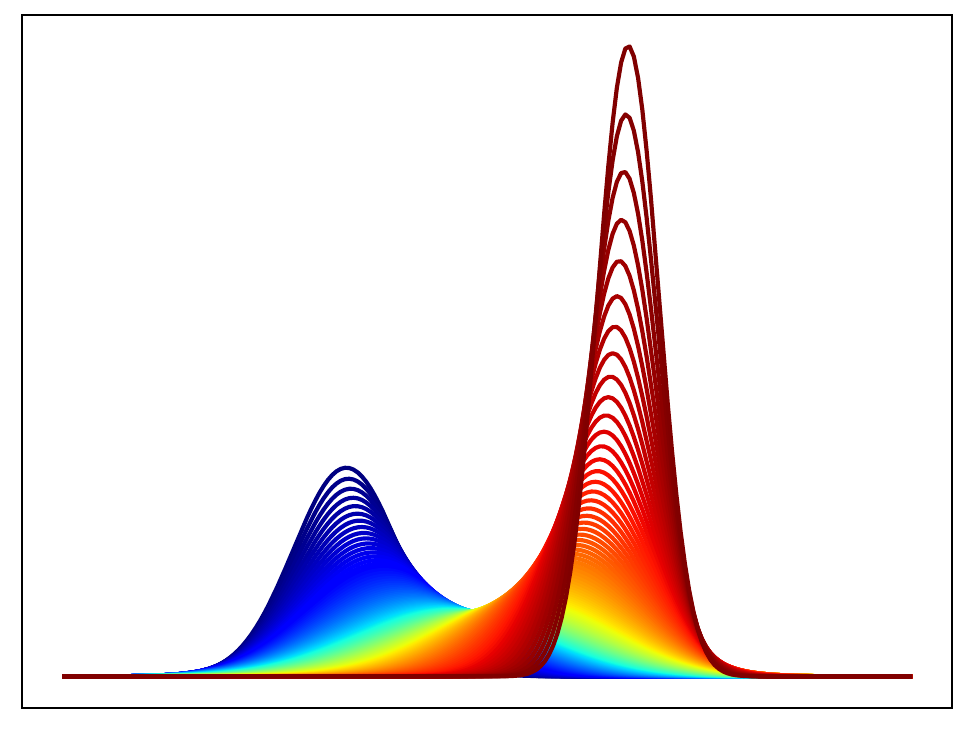}
    \caption*{$k=0.75$ \\ $\omega=1.4762$}
\end{subfigure}
\hfill
\begin{subfigure}{0.19\textwidth}
    \includegraphics[width=\textwidth]{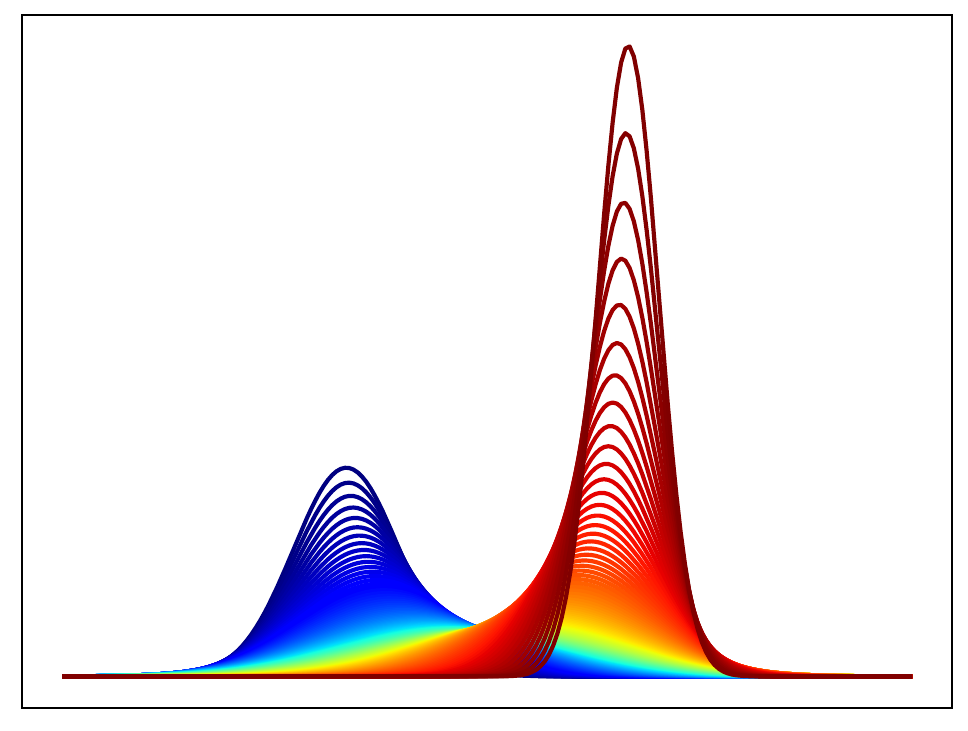}
    \caption*{$k=1$ \\ $\omega=1.9503$}
\end{subfigure}
\caption{Weighted averages between a curve connecting two measures and their geodesic. Five members of the family $\mu^{[k]}_t$ of~\eqref{eqn:Gaussian_weighted_family} are illustrated with their respective optimality numbers.}
\label{fig:Gaussian_averages}
\end{figure}

\begin{figure}[H]
    \centering
    \includegraphics[width=0.5\linewidth]{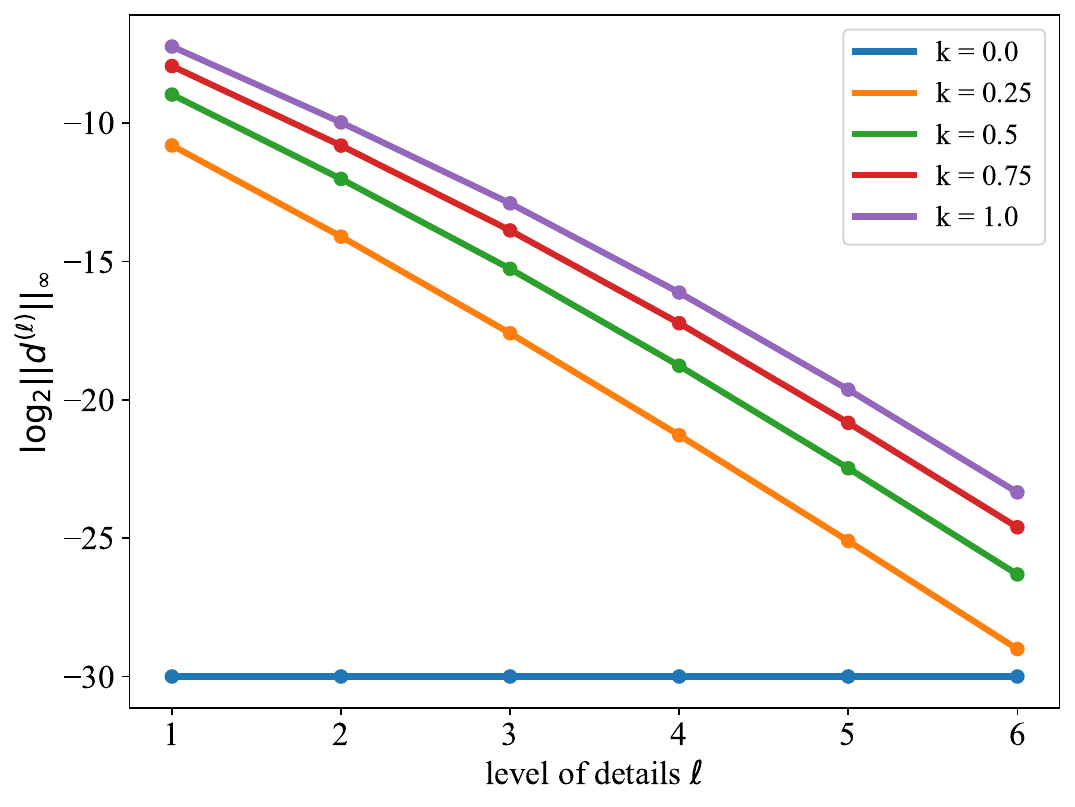}
    \caption{Maximal error against different detail layers $\ell$ on the logarithmic scale. The geometric decay of the maximal norm of the detail coefficient of five members of the family $\mu^{[k]}_t$ of~\eqref{eqn:Gaussian_weighted_family}.}
    \label{fig:Gaussian_averages_decay}
\end{figure}

\subsection{Curves of point clouds}

In this section, we demonstrate the elementary multiscale transform for discrete measures with free support on an example from physics. Sequences in this subsection form a point cloud that evolves with time according to a vector field. From a computational standpoint, each iteration of the multiscale decomposition~\eqref{eqn:elementary_multiscaling} requires solving several Kantorovich problems, each of which is a linear program whose complexity depends on the number of atoms of the involved measures. In the experiments of this section and the following one, all measures are supported on $10$ atoms, making each linear program small and the overall transform computationally inexpensive.

The electric field $v:\mathbb{R}^2\rightarrow\mathbb{R}^2$ induced by a positive charge $+q$ located at $(-1, 0)$ and a negative charge $-q$ located at $(1, 0)$ is given by Coulomb's law as
\begin{equation}~\label{eqn:electric_field}
    v(x,y) = \frac{1}{r_+^{3/2}}
    \begin{pmatrix}
        x+1 \\
        y
    \end{pmatrix}
    -\frac{1}{r_-^{3/2}}
    \begin{pmatrix}
        x-1 \\
        y
    \end{pmatrix}, \quad (x,y)\in\mathbb{R}^2,
\end{equation}
up to a constant depending on $q$ which we treat as $1$ for convenience, where $r_\pm=(x\pm1)^2+y^2$. Straightforward calculations of the Euclidean norm of $v(x,y)$ yield
\begin{align}~\label{eqn:electric_norm}
    \|v(x,y)\|^2 = \frac{\bigg[ y^2 \left(r_+ ^{3/2} - r_-^{3/2} \right)\bigg]^2+\bigg[(x - 1)r_+^{3/2} - (x + 1)r_-^{3/2}\bigg]^2
}{r_+^3 r_-^3},
\end{align}
which tends to $\infty$ as $(x,y)\rightarrow(\pm1,0)$. In other words, a particle beginning its trajectory from a point close to, say, the positive charge at $(-1, 0)$, would be pushed farther from the charge within a short fixed time interval. The closer the particle, the farther its location is by the next timestep. In contrast, particles moving along the field~\eqref{eqn:electric_field} in a large Euclidean distance from the origin would be less affected by the charges since $\|v(x,y)\|$ of~\eqref{eqn:electric_norm} tends to $0$ as $\|(x, y)\|\rightarrow\infty$.

We study the evolution of a point cloud in $\mathbb{R}^2$ along the field~\eqref{eqn:electric_field} with respect to time. Sequences in the Wasserstein space $\mathcal{P}_2(\mathbb{R}^2)$ in this setting would be samples of curves where each element encodes a finite set of distinct points. These curves, together with the field~\eqref{eqn:electric_field}, must satisfy the continuity equation~\eqref{eqn:continuity_equation}. To make this problem suitable with the free support measures from the optimal transport theory, we assume that the probability distribution on each cloud is uniform, and is time-invariant across the sequence.

We conduct and simulate two experiments. We first generate $10$ random points in the neighborhood of the point $(-2.5,1)\in\mathbb{R}^2$. Each generated point represents a particle. We track the trajectories of the particles along the electric field~\eqref{eqn:electric_field} with the prefixed timestep $0.15$. The final sequence of interest in the Wasserstein space $\mathcal{P}_2(\mathbb{R}^2)$ consists of $641$ discrete measures sampled along the flow induced by the field~\eqref{eqn:electric_field}. We decompose the resulting sequence with $6$ iterations of~\eqref{eqn:elementary_multiscaling}, leaving only $11$ points in the coarse approximation. Under these circumstances and parameters, some particles begin their movement near the positive charge. Hence, the detail coefficients of the sequence generated by the elementary multiscale transform would have relatively large values around their first timestep. This is indeed the case as Figure~\ref{fig:electric_pyramid_1} shows. Moreover, the decay in the detail coefficients across scales is explained through Theorem~\ref{thm:details_decay}.

\begin{figure}
\begin{subfigure}{0.38\textwidth}
    \includegraphics[width=\textwidth]{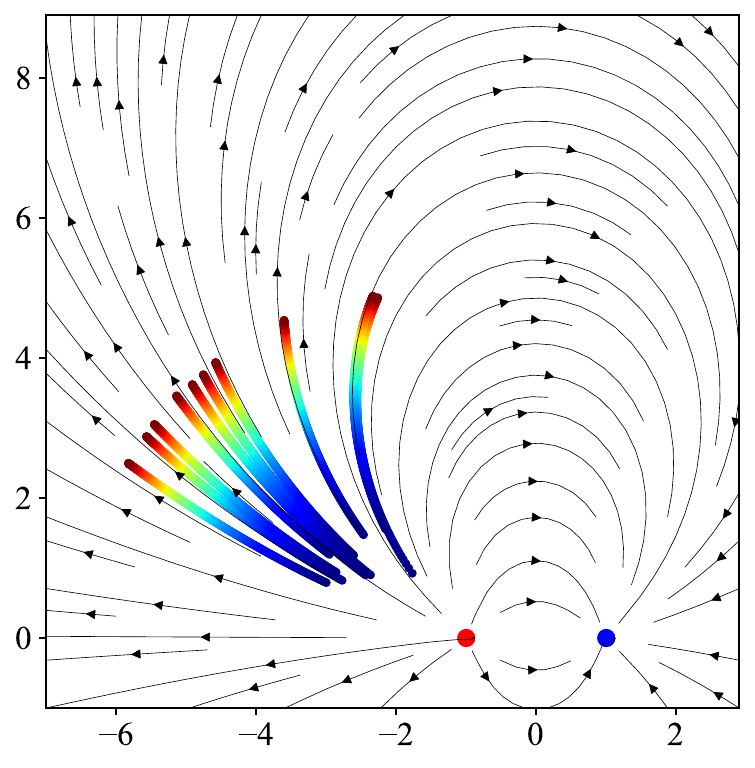}
\end{subfigure}
\begin{subfigure}{0.51\textwidth}
    \includegraphics[width=\textwidth]{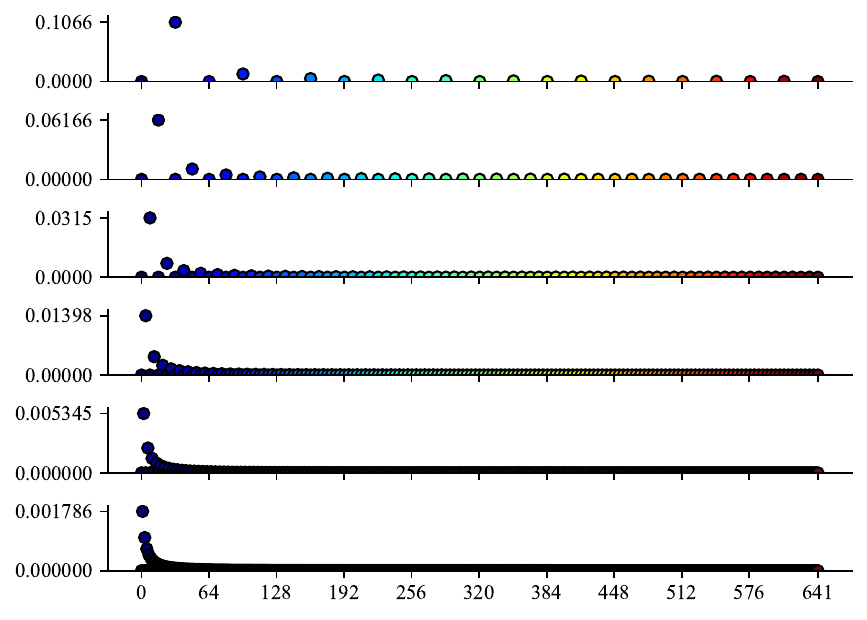}
\end{subfigure}
\caption{Multiscaling a flow of discrete measures in $\mathcal{P}_2(\mathbb{R}^2)$. On the left, the trajectories of the 10 particles along the electric field~\eqref{eqn:electric_field}. On the right, norms of the $6$ layers of detail coefficients obtained by the multiscaling~\eqref{eqn:elementary_multiscaling} of the sequence. Because some particles began their movement near the positive charge, the detail norms are salient on the left endpoint of the pyramid representation. The optimality number of the sequence is $\omega = 0.3209$.}
\label{fig:electric_pyramid_1}
\end{figure}

We note that following a vector field does not by itself make a curve of measures a geodesic in the Wasserstein space. Indeed, along a constant-speed geodesic between discrete measures, each atom travels along a straight line at constant speed, whereas the particle trajectories induced by~\eqref{eqn:electric_field} are curved, see Figure~\ref{fig:electric_pyramid_1}. The positive optimality number therefore reflects this deviation from geodesic motion, together with discretization effects.

Next, we contaminate the flow appearing in Figure~\ref{fig:electric_pyramid_1} with noise and test the multiscale transform of the resulting path. The noise is added to the atoms of the measures in the following sense. We start with the same $10$ points as before, but now, with every timestep, we calculate the vector field~\eqref{eqn:electric_field} and add to its two coordinates a noise that is normally distributed with $0$ mean and $0.1$ variance. Thanks to the additive noise, the sequence of measures now deviates from the original noise-free flow. This is manifested in large detail coefficients in the multiscale transform, which appears in Figure~\ref{fig:electric_pyramid_2} alongside the sequence itself. In contrast to Figure~\ref{fig:electric_pyramid_1}, note that there is no decay in the maximal detail coefficient across scales, this phenomenon further aligns with Lemma~\ref{lem:details_delta_bound} and Theorem~\ref{thm:details_decay}.

\begin{figure}[H]
\begin{subfigure}{0.38\textwidth}
    \includegraphics[width=\textwidth]{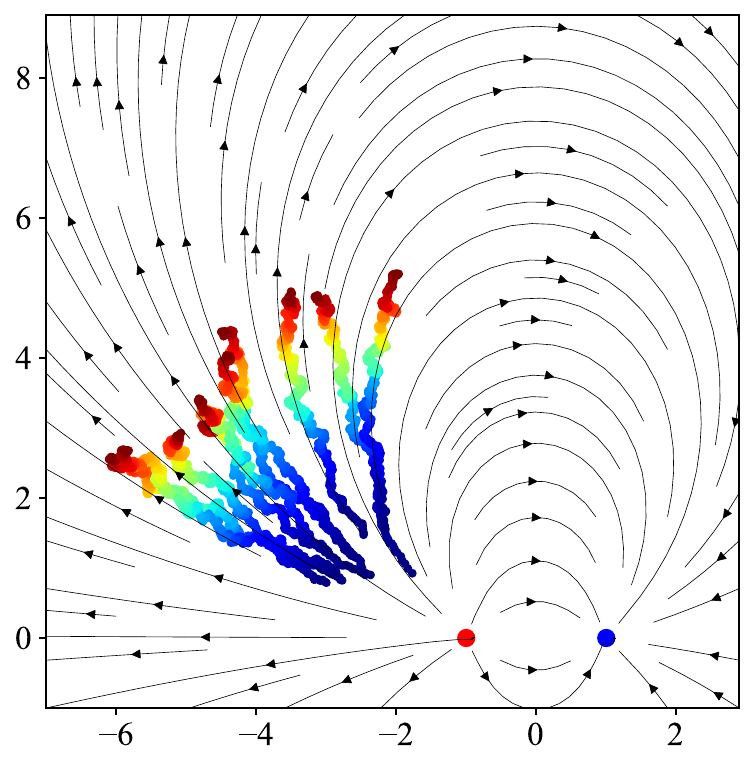}
\end{subfigure}
\begin{subfigure}{0.51\textwidth}
    \includegraphics[width=\textwidth]{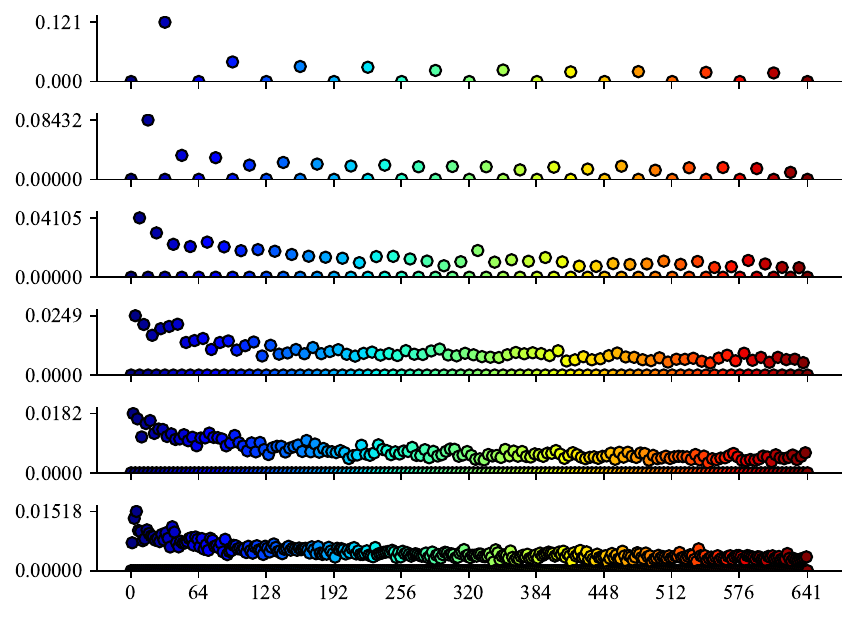}
\end{subfigure}
\caption{Multiscaling a noisy sequence of discrete measures in $\mathcal{P}_2(\mathbb{R}^2)$. On the left, the trajectories of the $10$ particles along the electric field~\eqref{eqn:electric_field}. On the right, norms of the $6$ layers of detail coefficients obtained by the multiscaling~\eqref{eqn:elementary_multiscaling} of the clouds. Because all particles are pushed farther from both charges, the detail coefficients exhibit geometric decay along the time axis. Due to the added noise, the maximal norm does not show a clear decay pattern. The optimality number of the analyzed sequence is $\omega = 4.6722$.}
\label{fig:electric_pyramid_2}
\end{figure}

The coarse approximations of the two sequences are shown in Figure~\ref{fig:electro_coarse}.

\begin{figure}[H]
\begin{subfigure}{0.35\textwidth}
    \includegraphics[width=\textwidth]{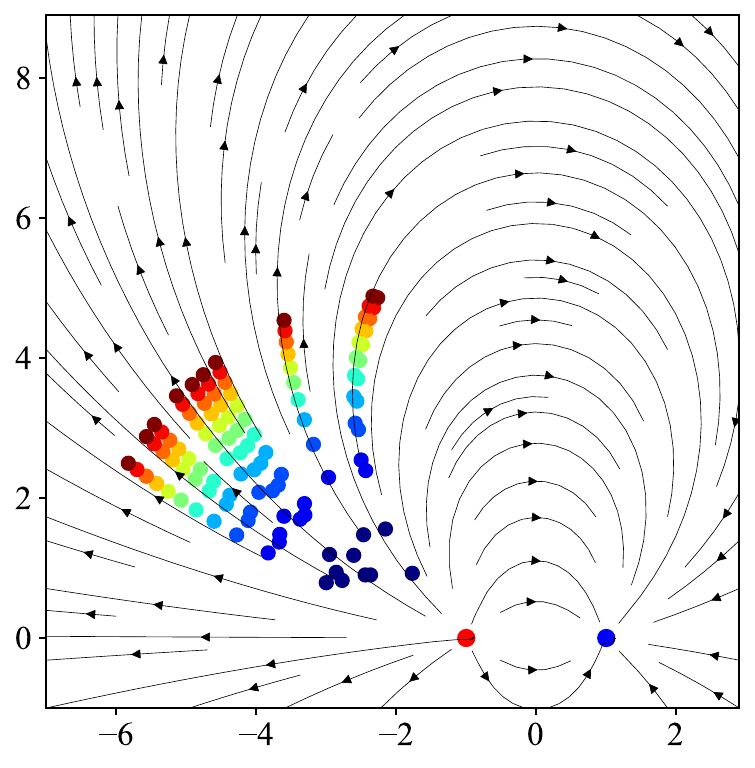}
\end{subfigure}
\begin{subfigure}{0.35\textwidth}
    \includegraphics[width=\textwidth]{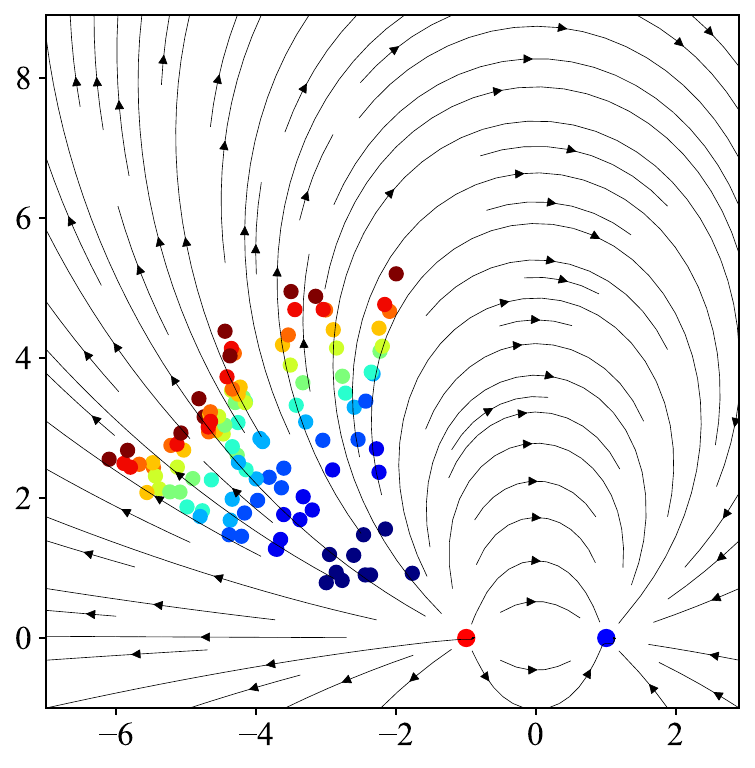}
\end{subfigure}
\caption{The coarse approximations of the curves appearing in Figures~\ref{fig:electric_pyramid_1} and~\ref{fig:electric_pyramid_2}. $11$ point clouds each consisting of $10$ atoms with uniform distribution.}
\label{fig:electro_coarse}
\end{figure}

The takeaway message of the two experiments presented in this section is as follows. The elementary multiscale transform~\eqref{eqn:elementary_multiscaling} can be used to study how smooth point clouds evolve over time. In particular, the faster the detail coefficients decay in scale, the smoother the flow of measures. Furthermore, locations where the sequence is affected by large vector fields, as seen in the continuity equation~\eqref{eqn:continuity_equation}, can be detected by large norms in the pyramid representation. Both insights are fully explained by our theoretical results presented in Section~\ref{sec:theoretical_results}.

\subsection{Learning dynamics of neural networks}

Our last numerical illustration is inspired by the deep learning theory. In this experiment, our objective is to show that our multiscale transform~\eqref{eqn:elementary_multiscaling} can be used to analyze different deep learning models and optimization methods. To this end, we track and study a sequence of discrete probability measures obtained by a deep neural network that solves a specific task.

We consider a convolutional neural network with $2346$ trainable weights for classifying the MNIST dataset~\cite{deng2012mnist}. The output layer consists of $10$ neurons that represent, due to the softmax activation function, a probability distribution over the class of digits $\{0,\dots,9\}$. We compile the neural network with the Adam optimizer, with the low learning rate $10^{-5}$, and the categorical cross-entropy loss function. While training the neural network on the training dataset, we calculate, at the end of each epoch, the mean probability distribution across all images with a given digit.

Due to the simplicity of the task, the learning rate is deliberately chosen to be small so that we can increase the number of epochs and thus create a sequence of discrete probability measures corresponding to a fine-grid parametrization. In our case, we consider $161$ epochs with batches of $128$. Furthermore, because all output measures share the same support, this makes the sequence analyzable by the operators presented in Section~\ref{sec:multiscaling_discrete}.

The following figure demonstrates the output mean probabilities for predicting the digit ``$3$'' on a heat map over the advancement of the epoch iterations. Next to the heat map is the multiscale representation of the measure sequence obtained by~\eqref{eqn:elementary_multiscaling}.

\begin{figure}[H]
\begin{subfigure}{0.46\textwidth}
    \includegraphics[width=\textwidth]{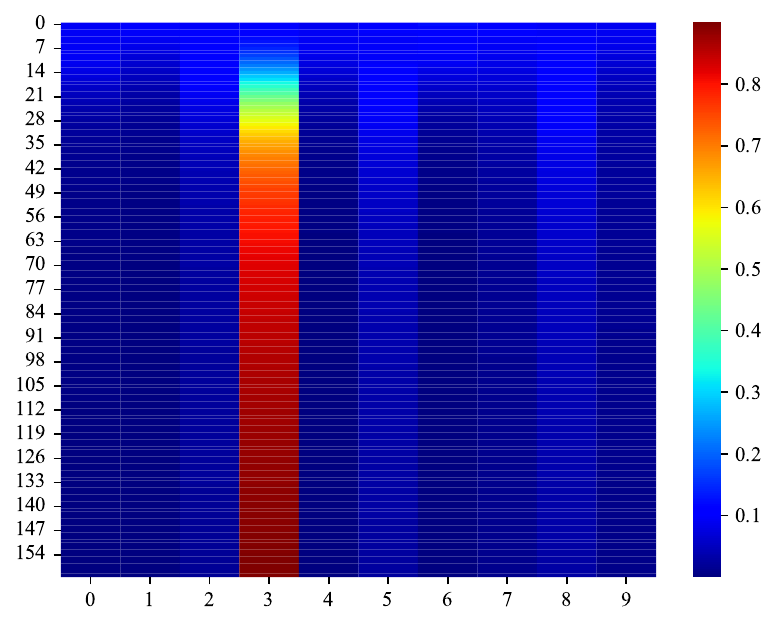}
\end{subfigure}
\hfill
\begin{subfigure}{0.51\textwidth}
    \includegraphics[width=\textwidth]{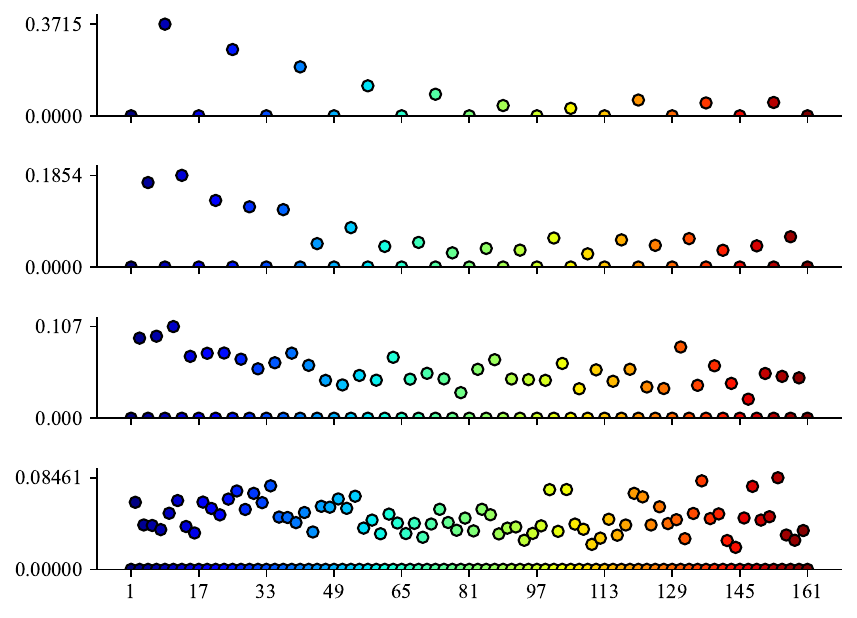}
\end{subfigure}
\caption{Analysis of the learning dynamics of a neural network. The heat map on the left depicts the mean probability of predicting the digit ``$3$'' by the end of each epoch, over 161 epochs. On the right, the norms of the detail coefficients~\eqref{eqn:elementary_multiscaling} over $4$ layers. In the early stages of learning, the distribution is more or less uniform, and as learning advances, the distribution converges to Dirac's measure over the specified digit. The convergence is apparent on the coarse scales of the details. The decay in the largest detail coefficient over scales indicates that the learning dynamics in $\mathcal{P}_2(\mathbb{R})$ are smooth.}
\label{fig:neural_network_pyramid}
\end{figure}

The decay of the largest detail coefficient across scales (see Figure~\ref{fig:neural_network_pyramid}) is consistent with the behavior predicted by Theorem~\ref{thm:details_decay} for sequences sampled from absolutely continuous curves, and indicates that the averaged output distributions of the network evolve regularly in $\mathcal{P}_2(\mathbb{R})$ throughout training. Furthermore, the measures converge to the Dirac measure at the digit ``$3$'' as learning progresses. This convergence is clear on the coarsest scale as the pyramid in Figure~\ref{fig:neural_network_pyramid} shows. In contrast, details on high scales do not exhibit an organized structure; these fluctuations are directly affected by the stochasticity involved in the optimization method.



The optimality number~\eqref{eqn:optimality_number} of the measure-valued sequence in Figure~\ref{fig:neural_network_pyramid} is $\omega = 8.788$; however, this value is most informative when compared against other sequences, such as those produced by different optimizers, learning rates, or network architectures, rather than interpreted in isolation.

Tailored to the nature of this experiment, the optimality number $\omega$ can be adjusted to capture different aspects of the learning dynamics, yielding a more informative value based on the desired parameters. For example, optimality can be calculated with respect to different batch sizes, learning rates, optimization methods, and network sizes. Moreover, the formula for optimality can include greater penalization for the early stages of learning, placing more emphasis on time regions where learning exhibits more dynamics, i.e., larger changes in weights with respect to the Wasserstein metric.

As a concluding remark, the experiment presented in this section hints at a broader potential of the multiscale framework in the machine learning community. In the context of flow-based generative models and flow matching methods~\cite{lipman2022flow}, a parametric velocity field $v_t$ is trained to generate a probability path $\mu_t$ connecting a latent distribution $\mu_0$ to a target data distribution $\mu_1$. The optimality number $\omega$ provides a principled scalar diagnostic for assessing how geodesically optimal such a path is in the Wasserstein space: a well-trained model approximating the optimal transport map would be expected to yield a low optimality number, while an inefficient or poorly trained model would yield a higher value. In this sense, the optimality number $\omega$ could serve as a model-agnostic evaluation metric for generative models, complementing existing metrics such as the Fr\'echet Inception Distance (FID)~\cite{heusel2017gans}.

\section*{Funding}

W. Mattar is partially supported by the Nehemia Levtzion Scholarship for Outstanding Doctoral Students from the Periphery (2023) and the DFG award 514588180. N. Sharon is partially supported by the NSF-BSF award 2024791, the BSF award 2024266, and the DFG award 514588180.

\bibliographystyle{plain}
\bibliography{references}

\section*{Appendix}

We prove the inequalities~\eqref{eqn:stability_inequalities}. Let $\mu,\nu\in\mathcal{P}_p(\mathbb{R}^d)$ for some $p>1$ and $\psi, \widetilde{\psi}:\mathbb{R}^d\to\mathbb{R}^d$ be two measurable Lipschitz maps. We have
\begin{equation*}
    W_p(\psi_\#\mu,\widetilde{\psi}_\#\mu) \leq \|\psi-\widetilde{\psi}\|_{L^p(\mu)} \quad \text{and} \quad W_p(\psi_\#\mu,\psi_\#\nu) \leq \|\psi\|_{\text{Lip}}W_p(\mu, \nu),
\end{equation*}
where $W_p$ is the Wasserstein distance~\eqref{eqn:Wasserstein_distance} and $\|\psi\|_{\text{Lip}}$ is the Lipschitz constant~\eqref{eqn:Lipschitz_norm} of $\psi$.

\begin{proof}
    For the first inequality, define $F:\mathbb{R}^d\to\mathbb{R}^d\times\mathbb{R}^d$ by $F(t)=(\psi(t), \widetilde{\psi}(t))$. The pushforward of $\mu$ via $F$ defines a measure $F_\#\mu$ over $\mathbb{R}^d\times\mathbb{R}^d$. Because $W^p_p$ of the left hand side is obtained by an optimal transport plan in $\Pi(\psi_\#\mu, \widetilde{\psi}_\#\mu)$, and since $F_\#\mu$ is in fact a transport plan, by change of variables we get
    \begin{equation*}
        W_p^p(\psi_\#\mu,\widetilde{\psi}_\#\mu)\leq \int_{\mathbb{R}^d\times \mathbb{R}^d} \|x-y\|^p d(F_\#\mu)(x,y) \leq \int_{\mathbb{R}^d} \|\psi(t)-\widetilde{\psi}(t)\|^pd\mu(t) = \|\psi-\widetilde{\psi}\|^p_{L^p(\mu)}.
    \end{equation*}

    Now, to prove the second inequality, we consider an optimal transport plan $\gamma\in\Pi(\mu,\nu)$. With such a measure we have $W^p_p(\mu,\nu)=\mathcal{J}_p(\gamma)$ where $\mathcal{J}_p$ is the functional~\eqref{eqn:wasserstein_functional}. Consider the mapping $H:\mathbb{R}^d\times \mathbb{R}^d\to\mathbb{R}^d\times \mathbb{R}^d$ given by $H(s,t)=(\psi(s), \psi(t))$. The pushforward of $\gamma$ via $H$ defines a transport plan $H_\#\gamma$ in $\Pi(\psi_\#\mu, \psi_\#\nu)$. Therefore,
    \begin{align*}
        W^p_p(\psi_\#\mu,\psi_\#\nu)& \leq\int_{\mathbb{R}^d\times \mathbb{R}^d}\|x-y\|^pd(H_\#\gamma)(x,y) \\
        & \leq \int_{\mathbb{R}^d\times \mathbb{R}^d}\|\psi(s)-\psi(t)\|^pd\gamma(s,t) \\
        & \leq \int_{\mathbb{R}^d\times \mathbb{R}^d} \|\psi\|^p_{\text{Lip}}\|s-t\|^pd\gamma(s,t) = \|\psi\|^p_{\text{Lip}}W^p_p(\mu,\nu).
    \end{align*}
\end{proof}
\end{document}